\documentclass[10pt, a4paper]{article}
\usepackage{color,url,latexsym,amssymb,amsthm,amsmath,amsxtra,amsfonts}
\usepackage{amsfonts}
\usepackage{amsbsy}
\usepackage{amssymb}
\usepackage{amsmath}
\usepackage{amsmath,amscd}
\usepackage{amsthm}
\usepackage{fancyhdr}

\theoremstyle{plain}
\newtheorem{theorem}{Theorem}[section]

\newtheorem{corollary}{Corollary}[section]
\newtheorem{proposition}{Proposition}[section]
\theoremstyle{definition}
\newtheorem{definition}{Definition}[section]
\newtheorem{remark}{Remark}[section]
\newtheorem{example}{Example}[section]

\begin{document}

\title{On almost complex Lie algebroids}
\author{Cristian Ida and Paul Popescu}
\date{}
\maketitle
\begin{abstract}
The almost complex Lie algebroids over smooth manifolds are introduced in the paper. In the first part we give some examples and we obtain a Newlander-Nirenberg type theorem on almost complex Lie algebroids. Next the almost Hermitian Lie algebroids and some related structures on the associated complex Lie algebroid are studied. For instance, we obtain that the $E$-Chern form of $E^{1,0}$ associated to an almost complex connection $\nabla$ on $E$ can be expressed in terms of the matrix $J_ER$, where $J_E$ is the almost complex structure of $E$ and $R$ is the curvature of $\nabla$. Also, we consider a metric product connection associated to an almost Hermitian Lie algebroid and we prove that the mean curvature section of $E^{0,1}$ vanishes and the second fundamental $2$--form section of $E^{0,1}$ vanishes iff the Lie algebroid is Hermitian. 
\end{abstract}

\medskip 
\begin{flushleft}
\strut \textbf{2010 Mathematics Subject Classification:} 32Q60, 53C15, 53C55, 22A22, 17B66.

\end{flushleft}

\section{Introduction and preliminaries}
\setcounter{equation}{0}
The Lie algebroids, \cite{Mack, Mack2}, are generalizations of Lie algebras and integrable distributions. In fact a Lie algebroid is an anchored vector bundle with a Lie bracket on module of sections. The cotangent bundle of a Poisson manifold has a natural structure of a Lie algebroid and between Poisson structures and Lie algebroids are many other connections, as for instance for every Lie algebroid structure on an anchored vector bundle there is a specific linear Poisson structure on the corresponding dual vector bundle and conversely. In the last decades the Lie algebroids are intensively studied by many authors, see for instance \cite{C-M,  Fe, H-M, Kos, L-M-M, Mack, Ma1,  Ma3}, from more points of view in the context of some different categories in differential geometry. Recently,  in the category of complex analytic geometry and of the $C^\infty$--foliated category, a general study of holomorphic Lie algebroids and  of foliated Lie and Courant algebroids is due to \cite{L-S-X, L-S-X2} and  \cite{Va3,Va2}, respectively. Also, in the category of Banach geometry the study of Lie algebroids was initiated in \cite{An2, An3} and  some significant developments are  given in \cite{Ca-Pe}. On the other hand, the study of Riemannian geometry of Lie algebroids is introduced and intensively studied in \cite{Bo} and a first treatment of (para) K\"{a}hlerian Lie algebroids can be found in \cite{L-T-W}. Other important structures as symplectic, hypersymplectic or Poisson structures on Lie algebroids are studied, see for instance \cite{A-C-N, I-M-D-M-P, Kos}. 

The notion of almost complex Lie algebroids over almost complex manifolds was introduced in \cite{B-R} as a natural extension of the notion of an almost complex manifold to that of an almost complex Lie algebroid. This generalizes the definition of an almost complex Poisson manifold given in \cite{C-F-I-U}, where some examples are also given. In \cite{B-R} this notion is used in order to obtain some cohomology theories  for skew-holomorphic Lie algebroids. Starting from the definition of an almost complex Lie algebroid, \cite{B-R}, but also from the general interest in the study of Lie algebroids, we have to consider that a general study of almost complex geometry in the almost complex Lie algebroids framework can be of some interest. However, for our purpose we will consider the almost complex Lie algebroids over smooth manifolds, not necessarily almost complex.

	The paper is organized as it follows. In the preliminary section we briefly recall some basic facts about Lie algebroids. For more, see for instance \cite{C-M, Fe, H-M, Mack, M, Po}. In the second section we define almost complex Lie algebroids over smooth manifolds, we present some examples and we obtain a Newlander-Nirenberg type theorem (Theorem \ref{t1}). Also the particular case when the base manifold is almost complex is discussed. In the third section we make a general approach about almost Hermitian Lie algebroids and sectional curvature of K\"{a}hlerian Lie algebroids over smooth manifolds in relation with corresponding notions from the geometry of almost complex manifolds \cite{G-O, Ya}. In particular we obtain a Schur type theorem for transitive K\"{a}hlerian Lie algebroids (Theorem \ref{Schur}). In the four section, the Hermitian metrics and linear connections compatible with such metrics on the associated complex Lie algebroid are studied and we present the Levi-Civita connection associated to such metrics. Also, we describe some $E$-Chern forms of $E^{1,0}$ associated to an almost complex connection $\nabla$ on $E$ in terms of the matrix $J_ER$, where $J_E$ is the almost complex structure of $E$ and $R$ is the curvature of $\nabla$ (Theorem \ref{Chern}). Finally, we consider a metric product connection associated to an almost Hermitian Lie algebroid and a $2$--form section for $E^{0,1}$ similar to the second fundamental form of complex distributions is studied in our setting. In particular, we prove that the mean curvature section of $E^{0,1}$ vanishes and the second fundamental $2$--form section of $E^{0,1}$ vanishes iff the Lie algebroid is Hermitian (Corollary \ref{curvsection}).

We notice that the present paper can be considered as an introduction of basic elements of almost complex geometry in the almost complex Lie algebroids framework. Some of these notions are continued in \cite{Po2} where almost complex Poisson structures on almost complex Lie algebroids are studied, but ohter problems related to almost complex geometry or complex (holomorphic) geometry as for instance: Laplacians or vanishing theorems are still open in the framework of Lie algebroids, as well as the study of anti-Hermitian (K\"{a}hlerian) Lie algebroids.  Also taking into account the role of almost complex geometry in the study of almost contact geometry the present notions can be useful in the study of almost contact or contact structures on Lie algebroids.

\subsection{Basic notions on Lie algebroids}
\begin{definition}
We say that $p:E\rightarrow M$ is an \textit{anchored} vector bundle if there exists a  vector bundle morphism $\rho:E\rightarrow TM$. The morphism $\rho$ will be called the anchor map.
\end{definition}
\begin{definition}
Let $(E,p,M)$ and $(E^{\prime},p^\prime,M^\prime)$ be two anchored vector bundles over the same base $M$ with the anchors $\rho:E\rightarrow TM$ and $\rho^\prime:E^\prime\rightarrow TM$. A morphism of anchored vector bundles over $M$ or a \textit{$M$--morphism of anchored vector bundles} between $(E,\rho)$ and $(E^\prime,\rho^\prime)$ is a morphism of vector bundles $\varphi:(E,p,M)\rightarrow(E^\prime,p^\prime,M)$ such that $\rho^\prime\circ\varphi=\rho$.
\end{definition}
The anchored vector bundles over the same base $M$ form a category. The objects are the pairs $(E,\rho_E)$ with $\rho_E$ the anchor of $E$ and a morphism $\phi:(E,\rho_E)\rightarrow(F,\rho_F)$ is a vector bundle morphism $\phi:E\rightarrow F$ which verifies the condition $\rho_F\circ\phi=\rho_E$.

Let $p:E\rightarrow M$ be an anchored vector bundle with the anchor $\rho:E\rightarrow TM$ and the induced morphism $\rho_E:\Gamma(E)\rightarrow \mathcal{X}(M)$. Assume there exists defined a bracket $[\cdot,\cdot]_E$ on the space $\Gamma(E)$ that provides a structure of real Lie algebra on $\Gamma(E)$.
\begin{definition}
The triplet $(E,\rho_E,[\cdot,\cdot]_E)$ is called a \textit{Lie algebroid} if
\begin{enumerate}
\item[i)] $\rho_E:(\Gamma(E,[\cdot,\cdot]_E)\rightarrow(\mathcal{X}(M),[\cdot,\cdot])$ is a Lie algebra homomorphism, that is

$\rho_E([s_1,s_2]_E)=[\rho_E(s_1),\rho_E(s_2)]$;

\item[ii)] $[s_1,fs_2]_E=f[s_1,s_2]_E+\rho_E(s_1)(f)s_2$, for every $s_1,s_2\in\Gamma(E)$ and $f\in C^{\infty}(M)$.
\end{enumerate}

\end{definition}
A Lie algebroid $(E,\rho_E,[\cdot,\cdot]_E)$ is said to be \textit{transitive}, if $\rho_E$ is surjective.

There exists a canonical cohomology theory associated to a Lie algebroid $(E,\rho_E,[\cdot,\cdot]_E)$ over a smooth manifold $M$. The space $C^{\infty}(M)$ is a $\Gamma(E)$-module relative to the representation $\Gamma(E)\times C^{\infty}(M)\rightarrow C^{\infty}(M),\,\,(s,f)\mapsto\rho_E(s)f$.

Following the well-known Chevalley-Eilenberg cohomology theory \cite{C-E}, we can introduce a cohomology complex associated to the Lie algebroid as follows. A $p$-linear mapping $\omega:\Gamma(E)\times\ldots\times\Gamma(E)\rightarrow C^{\infty}(M)$ is called a $C^{\infty}(M)$-valued $p$-cochain. Let $\mathcal{C}^p(E)$ denote the vector space of these cochains. The operator $d_E:\mathcal{C}^p(E)\rightarrow \mathcal{C}^{p+1}(E)$ given by
\begin{equation}
\begin{array}{ll}
d_E\omega (s_0, \ldots, s_r)=\sum\limits_{i=0}^r(-1)^i\rho _E(s_i)(\omega (s_0, \ldots, \widehat{s_i}, \ldots, s_r)) \\ 
\,\, & \,\, \\
\,\,\,\,\,\,\,\,\,\,\,\,\,\,\,\,\,\,\,\,\,\,\,\,\,\,\,\,\,\,\,\,\,\,\,\,\,\,\,\,\,\,\,\,\,\,\,\,\,\,+\sum\limits_{i<j=1}^r(-1)^{i+j}
\omega ([s_i, s_j]_E, s_0, \ldots, \widehat{s_i}, \ldots, \widehat{s_j}, \ldots, s_r) 
\end{array}
\label{I11}
\end{equation}
for $\omega\in \mathcal{C}^p(E)$ and $s_0,\ldots,s_p\in\Gamma(E)$, defines a coboundary since
$d_E\circ d_E=0$. Hence, $(\mathcal{C}^p(E),d_E)$, $p\geq1$ is a differential complex and the
corresponding cohomology spaces are called the cohomology groups of $\Gamma(E)$ with coefficients in $C^{\infty}(M)$. We notice that if $\omega\in \mathcal{C}^p(E)$ is  skew-symmetric and $C^{\infty}(M)$-linear, then $d_E\omega$ also is skew-symmetric.
From now on, the subspace of skew-symmetric and $C^{\infty}(M)$-linear cochains of the space $\mathcal{C}^p(E)$ will be denoted by $\Omega^p(E)$ and its elements will be called \textit{$p$--forms} on $E$. The \textit{Lie algebroid cohomology} $H^p(E)$ of $(E,\rho_E,[\cdot,\cdot]_E)$ is the cohomology of the subcomplex $(\Omega^p(E),d_E)$, $p\geq 1$.

\begin{definition}
Let $(E,\rho_E,[\cdot,\cdot]_E)$ and $(E^\prime,\rho_{E^\prime},[\cdot,\cdot]_{E^\prime})$ be two Lie algebroids over $M$. A
\textit{morphism of Lie algebroids} over $M$, is a morphism $\varphi:(E,\rho_E)\rightarrow(E^\prime,\rho_{E^\prime})$ of anchored
vector bundles with property that:
\begin{equation}
\label{I1}
d_E\circ \varphi^*=\varphi^*\circ d_{E^{\prime}},
\end{equation}
where $\varphi^*:\Omega^p(E^{\prime})\rightarrow\Omega^p(E)$ is defined by 
\begin{displaymath}
(\varphi^*\omega^{\prime})(s_1,\ldots,s_p)=\omega^{\prime}(\varphi(s_1),\ldots,\varphi(s_p)),\,\omega^{\prime}\in\Omega^p(E^{\prime}),\,s_1,\ldots,s_p\in\Gamma(E).
\end{displaymath}
We also say that $\varphi$ is a \textit{$M$--morphism of Lie algebroids}.
\end{definition}
Alternatively, we say that $\varphi:(E,\rho_E)\rightarrow(E^\prime,\rho_{E^\prime})$ is a \textit{$M$--morphism of Lie algebroids} if
$\varphi\left([s_1,s_2]_E\right)=\left[\varphi(s_1),\varphi(s_2)\right]_{E^\prime}\,,\,\,\forall \,s_1,s_2\in\Gamma(E)$.

The Lie algebroids over the same manifold $M$ and all $M$--morphisms of Lie
algebroids form a category, which is, via a forgetful functor, a subcategory of the category of anchored vector bundles over $M$.

If we consider $(x^{i})$, $i=1,\ldots,n$  a local coordinates system on $U\subset M$ and $\{e_a\}$, $a=1,\ldots,m$  a local basis of sections on the bundle $E$ over $U$, where $\dim M=n$ and ${\rm rank}\,E=m$, then $(x^{i},y^{a})$, $i=1,\ldots,n$, $a=1,\ldots,m$ are local coordinates on $E$. 
In a such local coordinates system, the anchor $\rho_E$ and the Lie bracket $[\cdot,\cdot]_E$ are expressed by the smooth functions $\rho^{i}_a$ and $C^a_{bc}$, namely
\begin{equation}
\label{I2}
\rho_E(e_a)=\rho^{i}_a\frac{\partial}{\partial x^{i}}\,\,{\rm and}\,\,[e_a,e_b]_E=C^c_{ab}e_c\,,\,i=1,\ldots,n,\,a,b,c=1,\ldots,m.
\end{equation}
The functions $\rho^{i}_a\,,\,C^{a}_{bc}\in C^\infty(M)$ given by the above relations are called the \textit{structure functions} of Lie algebroid $(E,\rho_E,[\cdot,\cdot]_E)$ in the given local coordinates system and their verify the following relations:
\begin{equation}
\label{I3}
\rho^j_a\frac{\partial\rho^{i}_b}{\partial x^j}-\rho^j_b\frac{\partial\rho^{i}_a}{\partial x^j}=\rho^{i}_cC^c_{ab}\,,\,C^c_{ab}=-C^c_{ba}\,,\,\sum_{cycl(a,b,c)}\left(\rho^{i}_a\frac{\partial C^d_{bc}}{\partial x^{i}}+C^{e}_{ab}C^d_{ce}\right)=0.
\end{equation}

The equations \eqref{I3} are called the \textit{structure equations} of Lie algebroid $(E,\rho_E,[\cdot,\cdot]_E)$.

\subsection{Linear connections. Torsion and curvature}

\begin{definition}
A \textit{linear connection} on the Lie algebroid $(E,\rho_E,[\cdot,\cdot]_E)$ over $M$, is a map $\nabla:\Gamma(E)\times\Gamma(E)\rightarrow\Gamma(E)$, $(s_1,s_2)\mapsto\nabla(s_1,s_2):=\nabla_{s_1}s_2\in\Gamma(E)$ such that:
\begin{enumerate}
\item[(1)] $\nabla$ is $\mathbb{R}$--bilinear;
\item[(2)] $\nabla_{fs_1}s_2=f\nabla_{s_1}s_2$, for all $f\in C^\infty(M)$ and $s_1,s_2\in\Gamma(E)$;
\item[(3)] $\nabla_{s_1}(fs_2)=(\rho_E(s_1)f)s_2+f\nabla_{s_1}s_2$, for all $f\in C^\infty(M)$ and $s_1,s_2\in\Gamma(E)$.
\end{enumerate}
\end{definition}
\begin{remark}
A linear connection on a Lie algebroid $(E,\rho_E,[\cdot,\cdot]_E)$ is in fact an $E$--connection in the vector bundle $E$. See for instance \cite{Fe} for the definition of an $E$-connection in a general vector bundle $F$.
\end{remark}
For every $s_1,s_2\in\Gamma(E)$, the section $\nabla_{s_1}s_2\in\Gamma(E)$ is called the \textit{covariant derivative of the section $s_2$ with respect to section $s_1$}. If $\{e_a\}$, $a=1,\ldots,m$ a local basis of sections on $E$ over $U\subset M$  then a linear connection $\nabla$ on $(E,\rho_E,[\cdot,\cdot]_E)$ is locally defined by a set of \textit{coefficient functions} $\Gamma^c_{ab}\in C^\infty(M)$ given by $\nabla_{e_a}e_b=\Gamma^c_{ab}e_c$. Then for every sections $s_1,s_2\in\Gamma(E)$ locally given by $s_1=s_1^{a}e_a$, $s_2=s_2^{b}e_b$, the covariant
derivative of the section $s_2$ with respect to section $s_1$ is given by $\nabla_{s_1}s_2=\left(s_1^{a}\rho_a^{i}\frac{\partial s_2^c}{\partial x^{i}}+\Gamma^c_{ab}s_1^{a}s_2^b\right)e_c$.

If $\nabla$ is a linear connection on the Lie algebroid $(E,\rho_E,[\cdot,\cdot]_E)$, the map $T:\Gamma(E)\times\Gamma(E)\rightarrow \Gamma(E)$ defined by
\begin{equation}
\label{I7}
T(s_1,s_2)=\nabla_{s_1}s_2-\nabla_{s_2}s_1-[s_1,s_2]_E\,,\,\,\forall\,s_1,s_2\in\Gamma(E),
\end{equation}
is called the \textit{torsion} of $\nabla$. We have that $T$ defined above is a tensor of type $(2,1)$ on $E$ which is $C^\infty(M)$--bilinear and antisymmetric.

Also, for a given linear connection $\nabla$ on the Lie algebroid $(E,\rho_E,[\cdot,\cdot]_E)$ we consider the map
$R:\Gamma(E)\times\Gamma(E)\times\Gamma(E)\rightarrow\Gamma(E)\,,\,(s_1,s_2,s_3)\mapsto R(s_1,s_2,s_3)=R(s_1,s_2)s_3$,
where the section $R(s_1,s_2)s_3$ is defined by
\begin{equation}
\label{I9}
R(s_1,s_2)s_3=\nabla_{s_1}\nabla_{s_2}s_3-\nabla_{s_2}\nabla_{s_1}s_3-\nabla_{[s_1,s_2]_E}s_3\,,\,\,\forall\,s_1,s_2,s_3\in\Gamma(E).
\end{equation}
It is easy to see that the map $R$ is $C^\infty(M)$--linear in every argument and it is antisymmetric with respect to the first two arguments, that is $R(s_1,s_2,s_3)=-R(s_2,s_1,s_3)\,,\,\,\forall\,s_1,s_2,s_3\in\Gamma(E)$. The map $R$ defined by \eqref{I9} is called the \textit{curvature} of the linear connection $\nabla$.
\begin{remark}
A linear connection $\nabla$ the Lie algebroid $(E,\rho_E,[\cdot,\cdot]_E)$ can be viewed as a map denoted again by $\nabla$ from $\Gamma(E)$ in $\Omega^1(E)\otimes\Gamma(E)$. More exactly, if $\{e_a\}$, $a=1,\ldots,m$ is a local basis for the sections of $E$, with respect to this basis, we can associate to the connection $\nabla$ the matrix $\theta=(\theta^b_a)$, $a,b=1,\ldots,m$ with elements $1$--forms on $E$ such that $\nabla_se_a=\sum\limits_{b=1}^m\theta_a^b(s)\otimes e_b$, $a=1,\ldots,m$.
Similarly, for the curvature of $\nabla$ we can associate in local basis $\{e_a\}$, $a=1,\ldots,m$ a matrix $(R^b_a)$, $a,b=1,\ldots,m$ with elements $2$--forms on $E$.
\end{remark}

We notice that for a given linear connection $\nabla$ on a Lie algebroid the exterior derivative $d_E$ can be usually expressed in terms of covariant derivative with respect to $\nabla$ and torsion of $\nabla$. Also the usual Bianchi identities hold in the Lie algebroid framework, see for instance \cite{Fe}.

\section{Almost complex Lie algebroids}
\setcounter{equation}{0}

In this section we define almost complex Lie algebroids over smooth manifolds, we present some examples and we obtain a Newlander-Nirenberg type theorem. Also the particular case when the base manifold is almost complex is discussed.

\subsection{Basic definitions, examples and results}

Let us consider a smooth manifold $M$, not necessarily almost complex, and a Lie algebroid $(E,\rho_E,[\cdot,\cdot]_E)$ over $M$ such that ${\rm rank}\,E=2m$. 

\begin{definition}  An \textit{almost complex structure $J_E$} on $(E,\rho_E,[\cdot,\cdot]_E)$ is an endomorphism $J_E:\Gamma(E)\rightarrow \Gamma(E)$, over the identity, such that $J_E^2=-{\rm id_{\Gamma(E)}}$ and a Lie algebroid $(E,\rho_E,[\cdot,\cdot]_E,J_E)$ endowed with such a structure will be called an \textit{almost complex Lie algebroid}.
\end{definition}

\begin{example}
Let $(M,J)$ be an almost complex manifold. Then its tangent bundle $TM$ is an almost complex Lie algebroid over $M$ with anchor the identity of $TM$ and with usual Lie bracket of vector fields.  
\end{example}

\begin{example}
\label{e2.2}
Let $(M,J,\pi^{2,0})$ be an almost complex Poisson manifold, see \cite{C-F-I-U}. Then $(M,\pi)$ is a real Poisson manifold, where $\pi=\pi^{2,0}+\overline{\pi^{2,0}}$. Then $\left(T^*M,\pi^{\#},\{\cdot,\cdot\}_{\pi}, J^*\right)$ is an almost complex Lie algebroid, where $J^*$ is the natural almost complex structure induced by $J$, $\pi^{\#}:\Gamma(T^*M)\rightarrow\Gamma(TM)$ is defined by $\pi^{\#}(\alpha)(\beta)=\pi(\alpha,\beta)$ and  $\{\alpha,\beta\}_{\pi}=\mathcal{L}_{\pi^{\#}(\alpha)}\beta-\mathcal{L}_{\pi^{\#}(\beta)}\alpha-d\pi(\alpha,\beta)$.
\end{example}

\begin{example}
\label{e2.1}
(Complete lift to the prolongation of a Lie algebroid). For a Lie algebroid $(E,\rho_E,[\cdot,\cdot]_E)$ with ${\rm rank}\,E=m$ we can consider the prolongation of $E$, see \cite{H-M,  Ma1, Ma3}, which is a vector bundle $p_L:\mathcal{L}^p(E)\rightarrow E$ of ${\rm rank}\,\mathcal{L}^p(E)=2m$ which has a Lie algebroid structure over $E$. More exactly, $\mathcal{L}^p(E)$ is the subset of $E\times TE$ defined by $\mathcal{L}^p(E)=\{(u,z)\,|\,\rho_E(u)=p_*(z)\}$, where $p_*:TE\rightarrow TM$ is the canonical projection. The projection on the second factor $\rho_{\mathcal{L}^p(E)}:\mathcal{L}^p(E)\rightarrow TE$, given by $\rho_{\mathcal{L}^p(E)}(u,z)=z$ will be the anchor of the prolongation Lie algebroid $\left(\mathcal{L}^p(E),\rho_{\mathcal{L}^p(E)},[\cdot,\cdot]_{\mathcal{L}^p(E)}\right)$ over $E$. For a smooth function $f\in C^\infty(M)$ its \textit{complete} and \textit{vertical lift} to $E$, $f^c$ and $f^v$ respectively, are given by $f^c(u)=\rho_E(u)f$ and $f^v(u)=(f\circ p)(u)$ for every $u\in E$. According to \cite{Ma1,Ma3}, we can consider the \textit{vertical lift} $s^v$ and the \textit{complete lift} $s^c$ of a section $s\in\Gamma(E)$ as sections of $\mathcal{L}^p(E)$ as follows. The local basis of $\Gamma(\mathcal{L}^p(E))$ is given by $\left\{\mathcal{X}_a(u)=\left(e_a(p(u)),\rho_a^{i}\frac{\partial}{\partial x^{i}}|_u\right),\mathcal{V}_a=\left(0,\frac{\partial}{\partial y^{a}}\right)\right\}$, where $\left\{\frac{\partial}{\partial x^{i}},\frac{\partial}{\partial y^{a}}\right\}$, $i=1,\ldots,n=\dim M$, $a=1\ldots,m={\rm rank}\,E$, is the local basis on $TE$. Then, the vertical and complete lifts, respectively, of a section $s=s^{a}e_a\in\Gamma(E)$ are given by
\begin{displaymath}
s^v=s^{a}\mathcal{V}_a\,,\,s^c=s^{a}\mathcal{X}_a+\left(\rho_E(e_c)(s^{a})-C^{a}_{bc}s^b\right)y^c\mathcal{V}_a.
\end{displaymath}
In particular, $e_a^v=\mathcal{V}_a$ and $e_a^c=\mathcal{X}_a-C^b_{ac}y^c\mathcal{V}_b$. 

Now, if ${\rm rank}\,E=2m$ and $J_E$ is an almost complex structure on the Lie algebroid $(E,\rho_E,[\cdot,\cdot]_E)$, then $J_E^c$ is an almost complex structure on $\left(\mathcal{L}^p(E),\rho_{\mathcal{L}^p(E)},[\cdot,\cdot]_{\mathcal{L}^p(E)}\right)$ (because one of the properties of the complete lift is: for $p(T)$ a polynomial, then $p(T^c)=p(T)^c$), and moreover, $J_{E}^{c}s^{v}=(J_{E}s)^{v}$ and $J_{E}^{c}s^{c}=(J_{E}s)^{c}$. 
\end{example}
\begin{example}
\label{e2.4}
Let us consider the prolongation Lie algebroid $\left(\mathcal{L}^p(E),\rho_{\mathcal{L}^p(E)},[\cdot,\cdot]_{\mathcal{L}^p(E)}\right)$ over $E$ from Example \ref{e2.1}. Let  $\nabla$ a linear connection on the Lie algebroid $E$ (in particular a Riemannian Lie algebroid $(E,g)$ and the Levi-Civita connection, see the next section). Then, the connection $\nabla$ leads to a natural decomposition of $\mathcal{L}^p(E)$ into vertical and horizontal subbundles, namely $\mathcal{L}^p(E)=H\mathcal{L}^p(E)\oplus V\mathcal{L}^p(E)$, where $V\mathcal{L}^p(E)=\mathrm{span}\,\{\mathcal{V}_a\}$ and $H\mathcal{L}^p(E)=\mathrm{span}\,\{\mathcal{H}_a=\mathcal{X}_a-\Gamma^b_{ac}y^c\mathcal{V}_b\}$, where $\Gamma^b_{ac}(x)$ are the local coefficients of the linear
connection $\nabla$. We notice that, the above decomposition can be obtained also by a nonlinear (Ehresmann) connection, see for instance \cite{PoL1}. The horizontal lift $s^h$ of a section $s=s^{a}e_a\in\Gamma(E)$ to $\mathcal{L}^p(E)$ is locally given by 
$s^h=s^{a}\mathcal{H}_a=s^{a}(\mathcal{X}_a-\Gamma^b_{ac}y^c\mathcal{V}_b)$.

Every section $\sigma\in\Gamma(\mathcal{L}^p(E))$ can be written as $\sigma=\sigma^h+\sigma^v$, accordingly to previous decomposition. Then a natural almost complex structure on $\mathcal{L}^p(E)$ is defined by
\begin{displaymath}
J_{\mathcal{L}^p(E)}:\Gamma(\mathcal{L}^p(E))\rightarrow \Gamma(\mathcal{L}^p(E))\,,\,J_{\mathcal{L}^p(E)}(\sigma^h)=-\sigma^v\,,\,J_{\mathcal{L}^p(E)}(\sigma^v)=\sigma^h
\end{displaymath}
and $\left(\mathcal{L}^p(E),\rho_{\mathcal{L}^p(E)},[\cdot,\cdot]_{\mathcal{L}^p(E)}, J_{\mathcal{L}^p(E)}\right)$ is an almost complex Lie algebroid over $E$.
\end{example}

\begin{example}
Let $M$ be a differentiable manifold of dimension $2m + n$ endowed with a
codimension $n$ foliation $\mathcal{F}$ (then the dimension of $\mathcal{F}$  is $2m$). According to \cite{E-K2}, the foliation $\mathcal{F}$ is said to be \textit{complex} if it can be defined by an open cover $\{U_i\},\,i\in I$, of $M$ and diffeomorphisms $\phi_i:\Omega_i\times\mathcal{O}_i\rightarrow U_i$ (where $\Omega_i$ is an open polydisc in $\mathbb{C}^m$ and $\mathcal{O}_i$ is an open ball in $\mathbb{R}^n$) such that, for every pair $(i,j)\in I\times I$ with $U_i\cap U_j\neq \phi$, the coordinate change 
\begin{displaymath}
\phi_{ij}=\phi_j^{-1}\circ\phi_i:\phi_i^{-1}(U_i\cap U_j)\rightarrow\phi_j^{-1}(U_i\cap U_j)
\end{displaymath}
is of the form $(z^{'}, x^{'})=(\phi_{ij}^1(z,x), \phi_{ij}^2(x))$ with $\phi_{ij}^1(z,x)$ holomorphic in $z$ for $x$ fixed.

If we set $z^k=u^k+iv^k$, $k=1,\ldots,m$, then the \textit{almost complex structure along the leaves} $J_{\mathcal{F}}:T\mathcal{F}\rightarrow T\mathcal{F}$, is given by 
$J_{\mathcal{F}}(\frac
\partial {\partial u ^k})=\frac \partial {\partial v ^k}$, $J_{\mathcal{F}}(\frac \partial {\partial v^k})=-\frac \partial
{\partial u^k}$, $k=1,\ldots,m$.

Then $\left(T\mathcal{F},i_{\mathcal{F}},[\cdot,\cdot]_{\mathcal{F}},J_{\mathcal{F}}\right)$ is an almost complex Lie algebroid with anchor the inclusion $i_{\mathcal{F}}:T\mathcal{F}\rightarrow TM$ and the usual Lie bracket $[\cdot,\cdot]_{\mathcal{F}}$ of the vector fields tangent to $\mathcal{F}$.
\end{example}

\begin{example}
\label{e2.5}
(Direct product structure). The direct product of two Lie algebroids $\left(E_1,\rho_{E_1},[\cdot,\cdot]_{E_1}\right)$ over $M_1$ and $\left(E_2,\rho_{E_2},[\cdot,\cdot]_{E_2}\right)$ over $M_2$ is defined in, \cite{Mack2} pg. 155, as a Lie algebroid structure $E_1\times E_2\rightarrow M_1\times M_2$. Let us briefly recall this construction. The general sections of $E_1\times E_2$ are of the form $s=\sum (f_i\otimes s_i^1)\oplus\sum(g_j\otimes s_j^2)$, where $f_i,g_j\in C^\infty(M_1\times M_2)$, $s_i^1\in\Gamma(E_1)$, $s_j^2\in\Gamma(E_2)$, and the anchor map is defined by
\begin{displaymath}
\rho_E\left(\sum (f_i\otimes s_i^1)\oplus\sum(g_j\otimes s_j^2)\right)=\sum(f_i\otimes\rho_{E_1}(s_i^1))\oplus\sum(g_j\otimes\rho_{E_2}(s_j^2)).
\end{displaymath}
Imposing the conditions
\begin{displaymath}
[1\otimes s^1,1\otimes t^1]_E=1\otimes[s^1,t^1]_{E_1}\,,\,[1\otimes s^1,1\otimes t^2]_E=0,
\end{displaymath}
\begin{displaymath}
[1\otimes s^2,1\otimes t^2]_E=1\otimes[s^2,t^2]_{E_2}\,,\,[1\otimes s^2,1\otimes t^1]_E=0,
\end{displaymath}
for every $s^1,t^1\in\Gamma(E_1)$ and $s^2,t^2\in\Gamma(E_2)$, it follows that for $s=\sum (f_i\otimes s_i^1)\oplus\sum(g_j\otimes s_j^2)$ and $s^{\prime}=\sum (f_k^{\prime}\otimes s_k^{\prime 1})\oplus\sum(g^{\prime}_l\otimes s_l^{\prime 2})$, we have, using Leibniz condition, the following expression for bracket on $E=E_1\times E_2$:
\begin{eqnarray*}
[s,s^\prime]_E&=&\left(\sum f_if^\prime_k\otimes[s_i^1,s_k^{\prime 1}]_{E_1}+\sum f_i\rho_{E_1}(s_i^1)(f_k^\prime)\otimes s_k^{\prime 1}-\sum f_k^\prime\rho_{E_1}(s_k^{\prime 1})(f_i)\otimes s_i^1\right)\\
&&\oplus\left(\sum g_jg^\prime_l\otimes[s_j^2,s_l^{\prime 2}]_{E_2}+\sum g_j\rho_{E_2}(s_j^2)(g_l^\prime)\otimes s_l^{\prime 2}-\sum g_l^\prime\rho_{E_2}(s_l^{\prime 2})(g_j)\otimes s_j^2\right).
\end{eqnarray*}
Now, if $E_1$ and $E_2$ are endowed with almost complex structures $J_{E_1}$ and $J_{E_2}$, respectively, then an almost complex structure $J_E$ on $E=E_1\times E_2$ can be defined by 
\begin{equation}
\label{C1}
J_E(s)=\sum (f_i\otimes J_{E_1}(s_i^1))\oplus\sum(g_j\otimes J_{E_2}(s_j^2)).
\end{equation}
\end{example}

Complexifying the real vector bundle $E$ we obtain the complex vector bundle $E_{\mathbb{C}}:=E\otimes_{\mathbb{R}}\mathbb{C}\rightarrow M$ and by extending the anchor map and the Lie bracket $\mathbb{C}$--linearly, we obtain a complex Lie algebroid $(E_{\mathbb{C}},[\cdot,\cdot]_E,\rho_E)$ with the anchor map $\rho_E:\Gamma(E_{\mathbb{C}})\rightarrow\Gamma(TM_{\mathbb{C}})$, that is, a homomorphism of the complexified of corresponding Lie algebras, and $[s_1,fs_2]_E=f[s_1,s_2]_E+\rho_E(s_1)(f)s_2$, for every $s_1,s_2\in\Gamma(E_{\mathbb{C}})$ and $f\in C^\infty(M)_{\mathbb{C}}= C^{\infty}(M)\otimes_{\mathbb{R}}\mathbb{C}$. Also, extending $\mathbb{C}$--linearly the almost complex structure $J_E$, we obtain the almost complex structure $J_E$ on $E_{\mathbb{C}}$.

As usual, we have a splitting
\begin{displaymath}
E_{\mathbb{C}}=E^{1,0}\oplus E^{0,1}
\end{displaymath}
according to the eigenvalues $\pm i$ of $J_E$ on $E_{\mathbb{C}}$. We also have
\begin{equation}
\label{II1}
\Gamma(E^{1,0})=\{s-iJ_Es\,|\,s\in\Gamma(E)\}\,,\,\Gamma(E^{0,1})=\{s+iJ_Es\,|\,s\in\Gamma(E)\}.
\end{equation}
Similarly, we have the splitting
\begin{displaymath}
E^*_{\mathbb{C}}:=E^*\otimes_{\mathbb{R}}\mathbb{C}=(E^{1,0})^*\oplus (E^{0,1})^*
\end{displaymath}
according to the eigenvalues $\pm i$ of $J^*_E$ on $E^*_{\mathbb{C}}$, where $J_E^*$ is the natural almost complex structure induced on $E^*$. We also have
\begin{equation}
\label{II2}
\Gamma((E^{1,0})^*)=\{\omega-iJ^*_E\omega\,|\,\omega\in\Gamma(E^*)\}\,,\,\Gamma((E^{0,1})^*)=\{\omega+iJ^*_E\omega\,|\,\omega\in\Gamma(E^*)\}.
\end{equation}
We set
\begin{displaymath}
\bigwedge^{p,q}(E)=\bigwedge^p(E^{1,0})^*\otimes\bigwedge^q(E^{0,1})^* \,\,{\rm and}\,\,\Omega^{p,q}(E)=\Gamma\left(\bigwedge^{p,q}(E)\right).
\end{displaymath}
Then, the differential $d_E$ of the complex $\Omega^{\bullet}(E)=\bigoplus_{p,q}\Omega^{p,q}(E)$ splits into the sum
\begin{displaymath}
d_E=\partial^\prime_E+\partial_E+\overline{\partial}_E+\partial^{\prime\prime}_E,
\end{displaymath}
where
\begin{displaymath}
\partial^\prime_E:\Omega^{p,q}(E)\rightarrow\Omega^{p+2,q-1}(E)\,,\,\partial_E:\Omega^{p,q}(E)\rightarrow\Omega^{p+1,q}(E),
\end{displaymath}
\begin{displaymath}
\overline{\partial}_E:\Omega^{p,q}(E)\rightarrow\Omega^{p,q+1}(E)\,,\,\partial^{\prime\prime}_E:\Omega^{p,q}(E)\rightarrow\Omega^{p-1,q+2}(E).
\end{displaymath}

For an almost complex Lie algebroid $(E,\rho_E,[\cdot,\cdot]_E, J_E)$ we can consider the Nijenhuis tensor of $J_E$ defined by
\begin{equation}
\label{2.1}
N_{J_E}(s_1,s_2)=[J_Es_1,J_Es_2]_E-J_E[s_1,J_Es_2]_E-J_E[J_Es_1,s_2]_E-[s_1,s_2]_E\,,\,\forall\,s_1,s_2\in\Gamma(E).
\end{equation}

\begin{proposition}
If we consider $\{e^{a}\}$, $a=1,\ldots,2m$ the dual basis of $\{e_{a}\}$, $a=1,\ldots,2m$, then the Nijenhuis tensor $N_{J_E}$ is locally defined by $N_{J_E}=N^c_{ab}e_c\otimes e^{a}\otimes e^b$, and its local coefficients are given by
\begin{eqnarray*}
N^c_{ab}&=&\rho_b^{i}\frac{\partial J_a^d}{\partial x^{i}}J_d^c-\rho_a^{i}\frac{\partial J_b^e}{\partial x^{i}}J_e^c+\rho_d^{i}\frac{\partial J_b^c}{\partial x^{i}}J_a^d-\rho_e^{i}\frac{\partial J_a^c}{\partial x^{i}}J_b^e\\
&&+J_a^dJ_e^cC^{e}_{bd}-J_b^dJ_e^cC^{e}_{ad}+J_b^eJ_a^dC^{c}_{de}-C^c_{ab},
\end{eqnarray*} 
where $J_E=J_a^be_b\otimes e^{a}$ is the local expression of the almost complex structure $J_E$.
\end{proposition}

\begin{definition}
An almost complex structure $J_E$ on the Lie algebroid $(E,\rho_E,[\cdot,\cdot]_E)$ of ${\rm rank}\,E=2m$ is called \textit{integrable} if $N_{J_E}=0$.
\end{definition}

Now, using a standard procedure from almost complex geometry, \cite{G-O,Hs,Ya}, we can prove the following Newlander-Nirenberg type theorem:
\begin{theorem}
\label{t1}
For an almost complex Lie algebroid $(E,\rho_E,[\cdot,\cdot]_E, J_E)$ over a smooth manifold $M$ the following assertions are equivalent:
\begin{enumerate}
\item[(i)] If $s_1,s_2\in\Gamma(E^{1,0})$ then $[s_1,s_2]_E\in\Gamma(E^{1,0})$;
\item[(ii))] If $s_1,s_2\in\Gamma(E^{0,1})$ then $[s_1,s_2]_E\in\Gamma(E^{0,1})$;
\item[(iii)] $d_E\Omega^{1,0}(E)\subset \Omega^{2,0}(E)+\Omega^{1,1}(E)$ and $d_E\Omega^{0,1}(E)\subset \Omega^{1,1}(E)\oplus\Omega^{0,2}(E)$;
\item[(iv)] $d_E\Omega^{p,q}(E)\subset \Omega^{p+1,q}(E)+\Omega^{p,q+1}(E)$;
\item[(v)] the real Nijenhuis $N_{J_E}$ from \eqref{2.1} vanish, namely $J_E$ is integrable.
\end{enumerate}
\end{theorem}

The above Newlander-Nirenberg type theorem says that for any integrable almost complex structure $J_E$ on an almost complex Lie algebroid $(E,\rho_E,[\cdot,\cdot]_E)$ we have the usual decomposition
\begin{displaymath}
d_E=\partial_E+\overline{\partial}_E.
\end{displaymath}
From $d_E^2=d_E\circ d_E=0$ we obtain the following identities:
\begin{equation}
\label{II5}
\partial_E^2=\overline{\partial}_E^2=\partial_E\overline{\partial}_E+\overline{\partial}_E\partial_E=0.
\end{equation}
Hence, in this case we obtain a Dolbeault type Lie algebroid cohomology as the cohomology of the complex $(\Omega^{p,\bullet}(E),\overline{\partial}_E)$.

\begin{remark}
In \cite{B-R} the integrability of the almost complex structure $J_E$ is defined in order to obtain a reduction to a holomorphic Lie algebroid which is equivalent to vanishing of a suitable Nijenhuis tensor on $E$.
\end{remark}

\begin{example}
Let $\left( E^{\prime },\rho _{E^{\prime }},[\cdot ,\cdot ]_{E^{\prime}}\right) $ be a Lie algebroid over a base manifold $M^{\prime }$ and $(E,\rho _{E},$ $[\cdot ,\cdot ]_{E})$ be a Lie subalgebroid. Let us denote as $i:M\rightarrow M^{\prime }$ and $I:E\rightarrow E^{\prime \prime }=i^{\ast}E^{\prime}$ be the inclusions; notice that $i$ is an inclusion submanifold
and $I$ is the incusion of a vector subbundle over $M$. Consider a projector $\Pi $ in the fibers of $E^{\prime\prime}$ such that the image of $\Pi$ is $E$, i.e. $\Pi \left( E^{\prime \prime }\right) =E$. A such projector $\Pi $ can be the ortogonal projection according to a Riemannian metric $g^{\prime \prime }$ in the fibers of $E^{\prime \prime }$, that can be
induced particularly by a metric $g^{\prime \prime }$ in the fibers of $E^{\prime }$. We can define an anchor $\rho _{E^{\prime \prime }}=\rho_{E}\circ \Pi :E^{\prime \prime }\rightarrow TM$ and a $\rho _{E^{\prime\prime }}$--bracket 
\begin{equation}
\lbrack s_1^{\prime \prime },s_2^{\prime \prime }]_{E^{\prime \prime }}=[\Pi
\left( s_1^{\prime \prime }\right) ,\Pi \left( s_2^{\prime \prime }\right)
]_{E},\,\forall\,s_1^{\prime \prime },s_2^{\prime \prime }\in \Gamma \left(
E^{\prime \prime }\right) .  \label{form-br}
\end{equation}

It is easy to check that $\left( E^{\prime \prime },\rho _{E^{\prime \prime}},[\cdot ,\cdot ]_{E^{\prime \prime }}\right) $ is a Lie algebroid. Indeed, if $s_1^{\prime \prime },s_2^{\prime \prime }\in \Gamma \left( E^{\prime \prime
}\right) $, then 

\begin{eqnarray*}
\rho _{E^{\prime \prime }}\left( [s_1^{\prime \prime },s_2^{\prime\prime }]_{E^{\prime \prime }}\right) &=&\rho _{E}\circ \Pi ([\Pi \left(s_1^{\prime \prime }\right) ,\Pi \left( s_2^{\prime \prime }\right) ]_{E})=\rho _{E}([\Pi \left( s_1^{\prime \prime }\right) ,\Pi \left( s_2^{\prime \prime}\right) ]_{E})\\
&=&[\rho _{E}\circ \Pi \left( s_1^{\prime \prime }\right)
,\rho _{E}\circ \Pi \left( s_2^{\prime \prime }\right) ]_{TM}=[\rho
_{E^{\prime \prime }}\left( s_1^{\prime \prime }\right) ,\rho _{E^{\prime\prime }}\left( s_2^{\prime \prime }\right) ]_{TM},
\end{eqnarray*} 
where $[\cdot ,\cdot ]_{TM}$ denotes the usual Lie bracket. 

Also, $\sum\limits_{cycl.}[[s_1^{\prime \prime },s_2^{\prime \prime
}]_{E^{\prime \prime }},s_3^{\prime \prime }]_{E^{\prime \prime }}=\sum\limits_{cycl.}[[\Pi \left( s_1^{\prime \prime }\right) ,\Pi \left(s_2^{\prime \prime }\right) ]_{E},\Pi \left(s_3^{\prime \prime }\right) ]_{E}=0$, since $[\cdot ,\cdot ]_{E}$ is a bracket of a Lie algebroid.

If $J_{E^{\prime }}$ is an almost complex structure in the fibers of $E^{\prime }$, then its restriction to $M$ gives an almost complex structure $J_{E^{\prime \prime }}$ in the fibers of $E^{\prime \prime }$.

If $s_1^{\prime \prime },s_2^{\prime \prime }\in \Gamma \left( E^{\prime \prime}\right) $, then for every $s_1^{\prime },s_2^{\prime }\in \Gamma \left(E^{\prime }\right) $ that extend $s_1^{\prime\prime}$ and $s_2^{\prime\prime}$, respectively, then the
restriction of $[s_1^{\prime },s_2^{\prime }]_{E^{\prime }}$ to $M$ does not depend on the extensions and it defines a section in $\Gamma \left(E^{\prime \prime }\right) $, that we denote as $[s_1^{\prime \prime},s_2^{\prime \prime }]_{E^{\prime }}^{\prime }$; it is not a bracket, but a restriction of $[\cdot ,\cdot ]_{E^{\prime }}$ to $M$. We say that the projection $\Pi $ is \emph{flat} if $[\Pi \left( s_1^{\prime \prime }\right),\Pi \left( s_2^{\prime \prime }\right) ]_{E}=\Pi \left( \lbrack s_1^{\prime\prime },s_2^{\prime \prime }]_{E^{\prime }}^{\prime }\right) $. 

Let us suppose that the projection $\Pi $ is flat. If $J_{E^{\prime }}$ is integrable and $\Pi\circ J_{E^{\prime\prime}}=J_{E^{\prime\prime}}\circ\Pi$  then $J_{E^{\prime \prime }}$ is integrable as well. Indeed, 
\begin{eqnarray*}
N_{J_{E^{\prime\prime }}}(s_1^{\prime \prime },s_2^{\prime \prime })&=&[J_{E^{\prime \prime}}s_1^{\prime \prime },J_{E^{\prime \prime }}s_2^{\prime \prime }]_{E^{\prime\prime }}-[s_1^{\prime \prime },s_2^{\prime \prime }]_{E^{\prime \prime}}-J_{E^{\prime\prime}}[J_{E^{\prime \prime }}s_1^{\prime \prime },s_2^{\prime \prime }]_{E^{\prime\prime }}-J_{E^{\prime\prime}}[s_1^{\prime \prime },J_{E^{\prime \prime }}s_2^{\prime \prime}]_{E^{\prime \prime }}\\
&=&[\Pi J_{E^{\prime \prime }}s_1^{\prime \prime },\Pi J_{E^{\prime \prime }}s_2^{\prime \prime }]_{E}-[\Pi s_1^{\prime \prime },\Pi s_2^{\prime \prime }]_{E}-J_{E^{\prime\prime}}[\Pi J_{E^{\prime \prime }}s_1^{\prime \prime },\Pi s_2^{\prime \prime }]_{E}-J_{E^{\prime\prime}}[\Pi s_1^{\prime \prime },\Pi J_{E^{\prime \prime}}s_2^{\prime \prime }]_{E}\\
&=&\Pi ([J_{E^{\prime \prime }}s_1^{\prime \prime },J_{E^{\prime \prime }}s_2^{\prime \prime }]_{E^{\prime }}^{\prime }-[s_1^{\prime \prime },s_2^{\prime \prime }]_{E^{\prime }}^{\prime }-J_{E^{\prime\prime}}[J_{E^{\prime \prime }}s_1^{\prime \prime },s_2^{\prime \prime }]_{E^{\prime}}^{\prime }-J_{E^{\prime\prime}}[s_1^{\prime \prime },J_{E^{\prime \prime }}s_2^{\prime\prime }]_{E^{\prime }}^{\prime })\\
&=&\Pi ([J_{E^{\prime }}s_1^{\prime },J_{E^{\prime }}s_2^{\prime }]_{E^{\prime }}-[s_1^{\prime },s_2^{\prime}]_{E^{\prime }}-J_{E^{\prime}}[J_{E^{\prime }}s_1^{\prime },s_2^{\prime }]_{E^{\prime}}-J_{E^{\prime}}[s_1^{\prime },J_{E^{\prime }}s_2^{\prime }]_{E^{\prime }})_{|M}\\
&=&N_{J_{E^{\prime }}}\left( s_1^{\prime },s_2^{\prime }\right) _{|M}=0,
\end{eqnarray*}
where $s_1^{\prime \prime },s_2^{\prime \prime }\in \Gamma \left( E^{\prime \prime}\right) $ are the restrictions to $M$ of $s_1^{\prime },s_2^{\prime }\in \Gamma\left( E^{\prime }\right) $ and we have used that $N_{J_{E^{\prime}}}=0$.

An example of such a case is when the algebroids are the tangent spaces $E=TS^{2n-1}\subset E^{\prime }=TI\!\!R^{2n}$  and $\Pi $ is the orthogonal projection according to the euclidean metric (see \cite{Po2}). It can be easily proved that $\Pi $ is flat, thus the canonical complex structure on $I\!\!R^{2n}$ gives an integrable almost complex structure on the algebroid $E=S^{2n-1}\times I\!\!R^{2n}$. The anchor of $E$ is the orthogonal projection on the tangent hyperplane, along on $S^{2n-1}$, and the bracket is given by a similar formula (\ref{form-br}).

\end{example}

We notice that in more situations in this paper we will consider the case when $J_E$ is integrable.

Let us consider now $(E,\rho_E,[\cdot,\cdot]_E, J_E)$ and $(E^\prime,\rho_{E^\prime},[\cdot,\cdot]_{E^\prime},J_{E^\prime})$ be two almost complex Lie algebroids over a smooth manifold $M$. 

A $M$--morphism $\varphi$ of almost complex Lie algebroids is naturally extended by $\mathbb{C}$--linearity to  $\varphi:(E_{\mathbb{C}},\rho_E,J_E)\rightarrow(E^\prime_{\mathbb{C}},\rho_{E^\prime},J_{E^\prime})$ over $M$ and it is called \textit{almost complex} if 
\begin{equation}
\label{II6}
\varphi\circ J_E=J_{E^\prime}\circ\varphi.
\end{equation}

\begin{proposition}
If $\varphi:(E,\rho_E,J_E)\rightarrow(E^\prime,\rho_{E^\prime},J_{E^\prime})$ is a $M$--morphism of almost complex Lie algebroids over $M$, then the following assertions are equivalent:
\begin{enumerate}
\item[(i)] If $s_1\in\Gamma(E^{1,0})$ then $\varphi(s_1)\in\Gamma(E^{\prime 1,0})$;
\item[(ii)] If $s_1\in\Gamma(E^{0,1})$ then $\varphi(s_1)\in\Gamma(E^{\prime 0,1})$;
\item[(iii)] If $\omega^{\prime}\in\Omega^{p,q}(E^{'})$ then $\varphi^*\omega^{\prime}\in\Omega^{p,q}(E)$, where
\begin{displaymath}
(\varphi^*\omega^\prime)(s_1,\ldots,s_p,t_1,\ldots,t_q)=\omega^\prime(\varphi(s_1),\ldots,\varphi(s_p),\varphi(t_1),\ldots,\varphi(t_q))
\end{displaymath}
for any $s_1,\ldots,s_p\in\Gamma(E^{1,0})$ and $t_1,\ldots,t_q\in\Gamma(E^{0,1})$.
\item[(iv)] The morphism $\varphi$ is almost complex.
\end{enumerate}
\end{proposition}
\begin{proof}
It follows using a standard argument from the almost complex geometry.
\end{proof}

If we consider $\{e_a\}$, $a=1,\ldots,m$ be a local basis of $\Gamma(E^{1,0})$ and $\{e_{\overline{b}}=\overline{e_b}\}$, $b=1,\ldots,m$ be a local basis of $\Gamma(E^{0,1})$, then we have
\begin{displaymath}
[e_a,e_b]_E=C^c_{ab}e_c+C^{\overline{c}}_{ab}\overline{e_c}\,,\,[e_a,e_{\overline{b}}]_E=C^c_{a\overline{b}}e_c+C^{\overline{c}}_{a\overline{b}}\overline{e_c},
\end{displaymath}
\begin{displaymath}
[e_{\overline{a}},e_b]_E=C^c_{\overline{a}b}e_c+C^{\overline{c}}_{\overline{a}b}\overline{e_c}\,,\,[e_{\overline{a}},e_{\overline{b}}]_E=C^c_{\overline{a}\,\overline{b}}e_c+C^{\overline{c}}_{\overline{a}\,\overline{b}}\overline{e_c},
\end{displaymath}
where $\overline{C^{a}_{bc}}=C^{\overline{a}}_{\overline{b}\,\overline{c}}$, $\overline{C^{\overline{c}}_{ab}}=C^{c}_{\overline{a}\,\overline{b}}$, $\overline{C^{c}_{\overline{a}b}}=C^{\overline{c}}_{a\overline{b}}$ and $\overline{C^{c}_{a\overline{b}}}=C^{\overline{c}}_{\overline{a}b}$, since $\overline{[s_1,s_2]_E}=[\overline{s_1},\overline{s_2}]_E$, for every $s_1,s_2\in\Gamma(E_{\mathbb{C}})$ and they are antisymmetric in below indices.

\begin{remark}
By Newlander-Nirenberg theorem, if $J_E$ is integrable then $C^{\overline{c}}_{ab}=C^c_{\overline{a}\,\overline{b}}=0$.
\end{remark}

\begin{definition}
Let $(E,\rho_E,[\cdot,\cdot]_E, J_E)$ be an almost complex Lie algebroid over $M$. A $M$-diffeomorphism $\varphi:(E,\rho_E,J_E)\rightarrow(E,\rho_E,J_{E})$ such that $\varphi\circ J_E=J_E\circ\varphi$ is called an \textit{automorphism} of the almost complex structure $J_E$. 
\end{definition}
\begin{definition}
\label{infinitezimal}
A section $s_1\in\Gamma(E)$ is called an \textit{infinitesimal automorphism} of $J_E$ if and only if
\begin{equation}
\label{II7}
[s_1,J_E(s_2)]_E-J_E\left([s_1,s_2]_E\right)=0,
\end{equation}
for any section $s_2\in\Gamma(E)$. 
\end{definition}
It is easy to see that the set of all infinitesimal automorphisms of $J_E$ is a Lie subalgebra of the Lie algebra of sections of $E$. Also, the following remarks hold:
\begin{remark}
If $s_1$ is an infinitesimal automorphism of $J_E$ then
\begin{displaymath}
N_{J_E}(s_1,s_2)=[J_E(s_1),J_E(s_2)]_E-J_E([J_E(s_1),s_2]_E)\,,\,\,\forall\,s_2\in\Gamma(E).
\end{displaymath}
\end{remark}
\begin{remark}
If $N_{J_E}=0$ then $s$ and $J_E(s)$ are simultaneously infinitesimal automorphisms of $J_E$.
\end{remark}
Extending the Definition \ref{infinitezimal} for the sections of $E_{\mathbb{C}}$ with respect to the almost complex $J_E$ on $E_{\mathbb{C}}$, we obtain
\begin{proposition}
Let $\{e_a\}$, $a=1,\ldots,m$ be a local basis of $E^{1,0}$. If all this sections are infinitesimal automorhisms of $J_E$ on $E_{\mathbb{C}}$, then the sections $\{e_{\overline{b}}\}$, $b=1,\ldots,m$, that give a local basis of $E^{0,1}$, are also infinitesimal automorphisms of $J_E$ on $E_{\mathbb{C}}$.
\end{proposition}
\begin{proof}
If $\{e_a\}$, $a=1,\ldots,m$ are infinitesimal automorphisms of $J_E$ on $E_{\mathbb{C}}$, then from $[e_a,J_E(e_b)]_E=J_E([e_a,e_b]_E)$ we obtain $C^{\overline{c}}_{ab}=0$ and from $[e_a,J_E(e_{\overline{b}})]_E=J_E([e_a,e_{\overline{b}}]_E)$ we obtain $C^c_{a\overline{b}}=0$. Then, by conjugation, we have $C^{c}_{\overline{a}\,\overline{b}}=C^{\overline{c}}_{\overline{a}b}=0$ which implies $[e_{\overline{a}},J_E(e_b)]_E=J_E([e_{\overline{a}},e_b]_E)$ and $[e_{\overline{a}},J_E(e_{\overline{b}})]_E=J_E([e_{\overline{a}},e_{\overline{b}}]_E)$.
\end{proof}

In the end of this subsection we prove that if $J_E$ is an integrable almost complex structure on $E$ then $(E^{1,0},E^{0,1})$ is a matched pairs of complex Lie algebroids. The notion of matched pairs of Lie algebroids was introduced in \cite{Lu} and further studied \cite{Mack3, Mo} and other authors.
\begin{definition}
\label{dmp}
A \textit{matched pair} of Lie algebroids is a pair of (complex or real) Lie algebroids $E_1$ and $E_2$ over the same base manifold $M$, where $E_2$ is an $E_1$-module and $E_1$ is an $E_2$-module such that the following identities hold:
\begin{equation}
\label{mp1}
\left[\rho_{E_1}(s),\rho_{E_2}(t)\right]=-\rho_{E_1}\left(\nabla_ts\right)+\rho_{E_2}(\nabla_st),
\end{equation}
\begin{equation}
\label{mp2}
\nabla_s[t_1,t_2]_{E_2}=\left[\nabla_st_1,t_2\right]_{E_2}+\left[t_1,\nabla_st_2\right]_{E_2}+\nabla_{\nabla_{t_2}s}t_1-\nabla_{\nabla_{t_1}s}t_2,
\end{equation}
\begin{equation}
\label{mp3}
\nabla_t[s_1,s_2]_{E_1}=\left[\nabla_ts_1,s_2\right]_{E_1}+\left[s_1,\nabla_ts_2\right]_{E_1}+\nabla_{\nabla_{s_2}t}s_1-\nabla_{\nabla_{s_1}t}s_2,
\end{equation}
where $s,s_1,s_2\in\Gamma(E_1)$ and $t,t_1,t_2\in\Gamma(E_2)$. Here $\rho_{E_1}$ and $\rho_{E_2}$ are the anchor maps of $E_1$ and $E_2$, respectively, and $\nabla$ denotes both $E_1$-connection on $E_2$ and $E_2$-connection on $E_1$, respectively
\begin{displaymath}
\Gamma(E_1)\times\Gamma(E_2)\rightarrow\Gamma(E_2)\,,\,(s,t)\mapsto\nabla_st\,\,{\rm and}\,\,\Gamma(E_2)\times\Gamma(E_1)\rightarrow\Gamma(E_1)\,,\,(t,s)\mapsto\nabla_ts.
\end{displaymath}
\end{definition}
Let $(E,\rho_E,[\cdot,\cdot]_E,J_E)$ be an almost complex Lie algebroid over a smooth manifold $M$ such that $J_E$ is integrable and $E_{\mathbb{C}}=E^{1,0}\oplus E^{0,1}$ its complexification. We consider the natural projections $p^{1,0}:E_{\mathbb{C}}\rightarrow E^{1,0}$ and $p^{0,1}:E_{\mathbb{C}}\rightarrow E^{0,1}$ from $E_{\mathbb{C}}$ onto $E^{1,0}$ and $E^{0,1}$, respectively, given by
\begin{equation}
\label{m1}
p^{1,0}=\frac{1}{2}(I_E-iJ_E)\,,\,p^{0,1}=\frac{1}{2}(I_E+iJ_E).
\end{equation}
We have
\begin{proposition}
\label{pmp}
If $(E,\rho_E,[\cdot,\cdot]_E,J_E)$ is an almost complex Lie algebroid over a smooth manifold $M$ such that $J_E$ is integrable then $(E^{1,0},E^{0,1})$ is a matched pair, where the actions are given by
\begin{displaymath}
\nabla_{s^{0,1}}s^{1,0}=p^{1,0}\left[s^{0,1},s^{1,0}\right]_{E}\,\,{\rm and}\,\,\nabla_{s^{1,0}}s^{0,1}=p^{0,1}\left[s^{1,0},s^{0,1}\right]_E
\end{displaymath}
for every $s^{1,0}\in\Gamma(E^{1,0})$ and $s^{0,1}\in\Gamma(E^{0,1})$. 
\end{proposition}
\begin{proof}
It is sufficient to verify the Definition \ref{dmp} in the local bases $\{e_a\}$, $a=1\ldots,m$ of $\Gamma(E^{1,0})$ and $\{\overline{e_a}\}$, $a=1\ldots,m$ of $\Gamma(E^{0,1})$. If $J_E$ is integrable we have $C^{\overline{c}}_{ab}=C^{c}_{\overline{a}\,\overline{b}}=0$. Let us consider $\rho^{1,0}:E^{1,0}\rightarrow T_{\mathbb{C}}M$, $\rho^{0,1}:E^{0,1}\rightarrow T_{\mathbb{C}}M$ given by 
\begin{equation}
\label{mp4}
\rho^{1,0}=\frac{1}{2}(\rho_E-i\rho_E\circ J_E)\,,\,\rho^{0,1}=\frac{1}{2}(\rho_E+i\rho_E\circ J_E),
\end{equation}
and $[\cdot,\cdot]^{1,0}:\Gamma(E^{1,0})\times\Gamma(E^{1,0})\rightarrow\Gamma(E^{1,0})$, $[\cdot,\cdot]^{0,1}:\Gamma(E^{0,1})\times\Gamma(E^{0,1})\rightarrow\Gamma(E^{0,1})$ given by
\begin{equation}
\label{mp5}
[\cdot,\cdot]^{1,0}=\frac{1}{2}([\cdot,\cdot]_E-iJ_E\circ [\cdot,\cdot]_E)\,,\,[\cdot,\cdot]^{0,1}=\frac{1}{2}([\cdot,\cdot]_E+iJ_E\circ [\cdot,\cdot]_E).
\end{equation}
Then locally, we have
\begin{equation}
\label{mp6}
\rho^{1,0}(e_a)=\rho_E(e_a)\,,\,\rho^{0,1}(\overline{e_a})=\rho_E(\overline{e_a})\,,\,[e_a,e_b]^{1,0}=[e_a,e_b]_E\,,\,[e_{\overline{a}},e_{\overline{b}}]^{0,1}=[e_{\overline{a}},e_{\overline{b}}]_E
\end{equation}
and using the Lie algebroid structure of $E$ it follows that $\left(E^{1,0},\rho^{1,0},[\cdot,\cdot]^{1,0}\right)$ and $\left(E^{0,1},\rho^{0,1},[\cdot,\cdot]^{0,1}\right)$ are complex Lie algebroids over $M$.

Now, the fact that $(E^{1,0},E^{0,1})$ is matched pair follows by direct verification of the Definition \ref{dmp} in the local bases $\{e_a\}$ and $\{\overline{e_a}\}$, where for \eqref{mp2} and \eqref{mp3} we use the Jacobi identity
$\sum\limits_{(a,\overline{b},\overline{c})}\left[e_a,\left[e_{\overline{b}},e_{\overline{c}}\right]_E\right]_E=0$ and its conjugate.
\end{proof}

\subsection{Almost complex Lie algebroids over almost complex manifolds}
In \cite{B-R} is given the definition of almost complex Lie algebroids over almost complex manifolds, which generalizes the definition of almost complex Poisson manifolds, see \cite{C-F-I-U}. In this subsection we consider the almost complex Lie algebroids over almost complex manifolds and in the particular case when the base manifold is complex,  we describe some properties of the structure functions of the associated complex Lie algebroid.

Consider $M$ a $2n$--dimensional almost complex manifold with an almost complex structure $J_M:\Gamma(TM)\rightarrow \Gamma(TM)$ and $(E,\rho_E,[\cdot,\cdot]_E)$ be a Lie algebroid over $M$ with ${\rm rank}\,E=2m$.
\begin{definition}
\cite{B-R} An \textit{almost complex structure $J_E$} on $(E,\rho_E,[\cdot,\cdot]_E)$ is an endomorphism $J_E:\Gamma(E)\rightarrow \Gamma(E)$ such that $J_E^2=-{\rm id_{\Gamma(E)}}$ and $J_M\circ\rho_E=\rho_E\circ J_E$.
\end{definition}
A real Lie algebroid $(E,\rho_E,[\cdot,\cdot]_E)$ endowed with such a structure will be called \textit{almost complex Lie algebroid over $(M,J_M)$}.

We notice that by the relation $\rho_E\circ J_E=J_M\circ\rho_E$ the complexified anchor map $\rho_E:\Gamma(E_{\mathbb{C}})\rightarrow\Gamma(TM_{\mathbb{C}})$ has the properties: $\rho_E(E^{1,0})\subset\Gamma(T^{1,0}M)$ and $\rho_E(E^{0,1})\subset\Gamma(T^{0,1}M)$, where $TM_{\mathbb{C}}=T^{1,0}M\oplus T^{0,1}M$ is the complexified of the real tangent bundle $TM$, and $T^{1,0}M={\rm span}\,\{Z_i\}$ and $T^{0,1}M={\rm span}\,\{\overline{Z_i}\}$ are the tangent bundles of complex vector of type $(1,0)$ and of type $(0,1)$, respectively. Then we can write
\begin{displaymath}
\rho_E(e_a)=\rho_a^{i}(z)Z_i\,,\,\rho_E(\overline{e_a})=\rho_{\overline{a}}^{\overline{i}}(z)\overline{Z_i},
\end{displaymath}
where $\{e_a\}$, $a=1,\ldots,m$ is a local basis of $\Gamma(E^{1,0})$ and $\{e_{\overline{b}}=\overline{e_b}\}$, $b=1,\ldots,m$ is a local basis of $\Gamma(E^{0,1})$.

If $(M,J_M)$ is a complex manifold, then we can consider the complex coordinates $(z^1,\ldots,z^n)$ on $M$ and then $T^{1,0}M={\rm span}\,\{\frac{\partial}{\partial z^{i}}\}$ and $T^{0,1}M={\rm span}\,\{\frac{\partial}{\partial\overline{z}^{i}}\}$ are the holomorphic and antiholomorphic tangent subbundles, respectively. Then, the above relations becomes 
\begin{displaymath}
\rho_E(e_a)=\rho_a^{i}(z)\frac{\partial}{\partial z^{i}}\,,\,\rho_E(\overline{e_a})=\rho_{\overline{a}}^{\overline{i}}(z)\frac{\partial}{\partial \overline{z}^{i}}.
\end{displaymath}
\begin{proposition}
If $(M,J_M)$ is a complex manifold, then the structure functions of the complex Lie algebroid $(E_{\mathbb{C}},[\cdot,\cdot]_E,\rho_E)$ over $(M,J_M)$ satisfy the following relations:
\begin{displaymath}
\rho^j_a\frac{\partial\rho^{i}_b}{\partial z^j}-\rho^j_b\frac{\partial\rho^{i}_a}{\partial z^j}=\rho^{i}_cC^c_{ab}\,,\,C^{\overline{c}}_{ab}\rho_{\overline{c}}^{\overline{i}}=0\,,\,C^c_{a\overline{b}}\rho_c^{i}=-\rho_{\overline{b}}^{\overline{j}}\frac{\partial \rho_a^{i}}{\partial\overline{z}^j}\,,\,C^{\overline{c}}_{a\overline{b}}\rho_{\overline{c}}^{\overline{i}}=\rho_a^j\frac{\partial\rho_{\overline{b}}^{\overline{i}}}{\partial z^j},
\end{displaymath}
\begin{displaymath}
\rho^{\overline{j}}_{\overline{a}}\frac{\partial\rho^{\overline{i}}_{\overline{b}}}{\partial \overline{z}^j}-\rho^{\overline{j}}_{\overline{b}}\frac{\partial\rho^{\overline{i}}_{\overline{a}}}{\partial \overline{z}^j}=\rho^{\overline{i}}_{\overline{c}}C^{\overline{c}}_{\overline{a}\,\overline{b}}\,,\,C^{c}_{\overline{a}\,\overline{b}}\rho_{c}^{i}=0\,,\,C^{\overline{c}}_{\overline{a}b}\rho_{\overline{c}}^{\overline{i}}=-\rho_{b}^{j}\frac{\partial \rho_{\overline{a}}^{\overline{i}}}{\partial z^j}\,,\,C^{c}_{\overline{a}b}\rho_{c}^{i}=\rho_{\overline{a}}^{\overline{j}}\frac{\partial\rho_{b}^{i}}{\partial \overline{z}^j}.
\end{displaymath}
 \end{proposition}
\begin{proof}
Follows by direct calculations in the relation $\rho_E\left([e_a,e_b]_E\right)=[\rho_E(e_a),\rho_E(b)]$ and similarly for the other sections of $E_{\mathbb{C}}$.
\end{proof}

\begin{remark}
\label{rem 3.2}
Taking into account the relations $\rho_E\circ J_E=J_M\circ\rho_E$ and $\rho_E([s_1,s_2]_E)=[\rho_E(s_1),\rho_E(s_2)]$ we obtain
\begin{equation}
\label{II4}
\rho_E\left(N_{J_E}(s_1,s_2)\right)=N_{J_M}(\rho_E(s_1),\rho_E(s_2)),
\end{equation}
which says that if the almost complex Lie algebroid $(E,\rho_E,[\cdot,\cdot]_E, J_E)$ over $(M,J_M)$ is transitive and $J_E$ is integrable then $J_M$ is integrable, that is $M$ is a complex manifold.
\end{remark}

It is well known, see \cite{N-W,Va3}, that there is an one-to-one correspondence between Lie algebroid structures on the vector bundle $p:E\rightarrow M$ and specific Poisson structures on the total space of the corresponding dual bundle $E^*\rightarrow M$, called the \textit{dual Poisson structures}, defined by the following brackets of basic and fiber-linear functions (with respect to foliation by fibres of $E^*$):
\begin{equation}
\label{Pos1}
\{p^*f_1,p^*f_2\}=0\,,\,\{p^*f,l_s\}=-p^*(\rho_E(s)f)\,,\,\{l_{s_1},l_{s_2}\}=l_{[s_1,s_2]_E},
\end{equation}
where $f,f_1,f_2\in C^\infty(M)$, $s,s_1,s_2\in\Gamma(E)$ and $l_s$ is the evaluation of the fiber of $E^*$ on $s$. If $(x^{i})$ are local coordinates on $M$ and $(\zeta_a)$ are the fiber coordinates on $E^*$ with respect to the dual of a local basis $\{e_{a}\}$ of sections of $E$, then the corresponding Poisson bivector field is given by
\begin{equation}
\label{Pos2}
\Lambda=\frac{1}{2}C^{a}_{bc}\zeta_a\frac{\partial}{\partial\zeta_b}\wedge\frac{\partial}{\partial\zeta_c}+\rho_b^{i}\frac{\partial}{\partial\zeta_b}\wedge\frac{\partial}{\partial x^{i}}.
\end{equation}
Conversely, formulas \eqref{Pos2} together with \eqref{I2} produce an anchor $\rho_E$ and a bracket $[\cdot,\cdot]_E$ and the Poisson condition $[\Lambda,\Lambda]=0$ implies the Lie algebroid axioms. 

Now if $(M,J_M)$ is an almost complex manifold with $\dim M=2n$ and $(E,J_E)$ is an almost complex Lie algebroid over $(M,J_M)$, ${\rm rank}\,E=2m$, then the total space of $E^*$  has a natural structure of almost complex manifold with almost complex structure $\mathcal{J}$ induced naturally by $J_M$ and $J_E$. In this case, the real Poisson bivector field $\Lambda$ from \eqref{Pos2} admits the decomposition $\Lambda=\Lambda^{2,0}+\Lambda^{1,1}+\overline{\Lambda^{2,0}}$, according to bigraduation $\mathcal{V}^r(E^*,\mathcal{J})=\bigoplus_{p+q=r}\mathcal{V}^{p,q}(E^*,\mathcal{J})$ of multivector fields on $(E^*,\mathcal{J})$. A natural question to ask is if $\Lambda^{2,0}\in\mathcal{V}^{2,0}(E^*,\mathcal{J})$ defines an almost complex Poisson structure on $(E^*,\mathcal{J})$, that is  $[\Lambda^{2,0},\Lambda^{2,0}]=0$ and $[\Lambda^{2,0},\overline{\Lambda^{2,0}}]=0$. If $\Lambda^{1,1}=0$ then it is true when $\mathcal{J}$ is integrable or otherwise, if moreover the real bivector field $\Lambda^{\prime}=i(\Lambda^{2,0}-\overline{\Lambda^{2,0}})$ is also real Poisson, (see Proposition 3.2 from \cite{C-F-I-U}).

\section{Almost Hermitian Lie algebroids}
\setcounter{equation}{0}
In this section we make a general approach about almost Hermitian Lie algebroids and sectional curvature of K\"{a}hlerian Lie algebroids over smooth manifolds in relation with corresponding notions from the geometry of almost complex manifolds.

Firstly we notice that if the vector bundle $(E,p,M)$ is endowed with a Riemannian metric $g$ then, according to \cite{Bo, C-M}, there exists an unique linear connection $D$ in the Lie algebroid $(E,\rho_E,[\cdot,\cdot]_E)$ such that
$D$ is compatible with $g$ and it is torsion free. It is given by the formula
\begin{equation}
\begin{array}{ll}
2g_E(D_{s_1}s_2, s_3)=\rho _E(s_1)(g_E(s_2, s_3))+\rho _E(s_2)(g_E(s_1, s_3))-\rho _E(s_3)(g_E(s_1, s_2)) \\
\,\,\,\,\,\,\,\,\,\,\,\,\,\,\,\,\,\,\,\,\,\,\,\,\,\,\,\,\,\,\,\,\,\,\,\,\,\,\,\,\,\,\,\,\,\,+g_E([s_3,s_1]_E,s_2)+g_E([s_3,s_2]_E,s_1)+g_E([s_1,s_2]_E,s_3).
\end{array}
\label{III4}
\end{equation}
and its local coefficients are given  by
\begin{equation}
\label{III5}
\Gamma^{a}_{bc}=\frac{1}{2}g^{ad}\left(\rho_E(e_b)(g_{cd})+\rho_E(e_c)(g_{bd})-\rho_E(e_d)(g_{bc})+C^{e}_{dc}g_{eb}+C^{e}_{db}g_{ec}-C^{e}_{bc}g_{ed}\right),
\end{equation}
where $g_{ab}=g(e_a,e_b)$ and $(g^{ba})$ is the inverse matrix of $(g_{ab})$. The connection $D$ given by \eqref{III4} is called the Levi-Civita connection of $(E,\rho_E,[\cdot,\cdot]_E)$ endowed with a Riemannian metric $g$.

In the following we consider an almost complex Lie algebroid $(E,\rho_E,[\cdot,\cdot]_E,J_E)$ over a smooth manifold $M$.
\begin{definition}
A Hermitian metric on an almost complex Lie algebroid $(E,\rho_E,[\cdot,\cdot]_E, J_E)$ is a Riemannian metric $g$ on $E$ invariant by $J_E$, that is
\begin{equation}
g(J_E(s_1),J_E(s_2))=g(s_1,s_2)\,,\,\forall\,s_1,s_2\in\Gamma(E)
\label{VII1}
\end{equation}
and an almost complex Lie algebroid $(E,\rho_E,[\cdot,\cdot]_E, J_E)$ endowed with a Hermitian metric $g$ is called \textit{almost Hermitian Lie algebroid} and denote  $(E,\rho_E,[\cdot,\cdot]_E, J_E, g)$. Moreover if $N_{J_E}=0$ then it is called \textit{Hermitian Lie algebroid}.
\end{definition}
We notice that if an almost complex Lie algebroid $(E,\rho_E,[\cdot,\cdot]_E, J_E)$ admits a Riemannian metric $g$, then a Hermitian metric on $E$ can be usually defined by $h(s_1,s_2)=g(s_1,s_2)+g(J_E(s_1),J_E(s_2))$ for every $s_1,s_2\in\Gamma(E)$.

\begin{example}
\label{e3.1}
The complete lift $g^c$ to $\mathcal{L}^p(E)$ of almost Hermitian metric $g$ on $E$ is defined by
$g^c(s_1^c,s_2^c)=(g(s_1,s_2))^c$, and it is easy to see that $g^c(J_E^c(s_1^c),J_E^c(s_2^c))=g^c(s_1^c,s_2^c)$, that is $g^c$ is almost Hermitian metric on $\mathcal{L}^p(E)$ too with respect to $J_E^c$. Furthermore, since 
\begin{equation*}
[s_1^v,s_2^v]_{\mathcal{L}^p(E)}=0\,,\,[s_1^v,s_2^c]_{\mathcal{L}^p(E)}=[s_1,s_2]^v_{E}\,,\,[s_1^c,s_2^c]_{\mathcal{L}^p(E)}=[s_1,s_2]^c_{E},
\end{equation*}
see \cite{Ma}, it follows $N_{J_{E}^{c}}=\left(N_{J_{E}}\right) ^{c}$. Thus $(E,\rho_E,[\cdot,\cdot]_E, J_E,g)$ is a Hermitian Lie algebroid iff $(\mathcal{L}^p(E),\rho_{\mathcal{L}^p(E)},[\cdot,\cdot]_{\mathcal{L}^p(E)}, J^c_E,g^c)$ is also Hermitian.
\end{example}
\begin{example}
\label{e3.2}
Let us consider the almost complex Lie algberoid $\left(\mathcal{L}^p(E),\rho_{\mathcal{L}^p(E)},[\cdot,\cdot]_{\mathcal{L}^p(E)}, J_{\mathcal{L}^p(E)}\right)$  from Example \ref{e2.4}. Then, we can define the Sasaki type metric $g_{\mathcal{L}^p(E)}$ on $\mathcal{L}^p(E)$ by
\begin{equation*}
g_{\mathcal{L}^p(E)}(s_1^h,s_2^h)=g_{\mathcal{L}^p(E)}(s_1^v,s_2^v)=(g_E(s_1,s_2))^v\,,\,g_{\mathcal{L}^p(E)}(s_1^h,s_2^v)=0,
\end{equation*}
for every $s_1,s_2\in\Gamma(E)$ and it is easy to see that this metric is Hermitian with respect to $J_{\mathcal{L}^p(E)}$. Moreover since the Levi-Civita connection of the Riemannian metric $g$ on  $E$ is torsion free, then $J_{\mathcal{L}^p(E)}$ is integrable iff the horizontal bundle $H\mathcal{L}^p(E)$ is integrable.
\end{example}
\begin{example}
\label{e3.3}
If $(E_1,\rho_{E_1},[\cdot,\cdot]_{E_1}, J_{E_1}, g_1)$ and $(E_2,\rho_{E_2},[\cdot,\cdot]_{E_2}, J_{E_2}, g_2)$ are two almost Hermitian Lie algebroids then the metric 
\begin{eqnarray*}
&&g\left(\left(\sum(f_i\otimes s_i^1)\oplus\sum(g_j\otimes s_j^2)\right),\sum (f_k^{\prime}\otimes s_k^{\prime 1})\oplus\sum(g^{\prime}_l\otimes s_l^{\prime 2})\right)\\
&&=\sum f_if_k^\prime\otimes g_{1}(s_i^1,s_k^{\prime 1})\oplus\sum g_jg_l^{\prime}\otimes g_{2}(s_j^2,s_{l}^{\prime 2})
\end{eqnarray*}
is a Hermitian metric on  the almost complex Lie algebroid given by direct product $E=E_1\times E_2$ from Example \ref{e2.5} with respect to almost complex structure $J_E$ given by \eqref{C1}.
\end{example}

A Hermitian metric $g$ on an almost complex Lie algebroid $(E,\rho_E,[\cdot,\cdot]_E, J_E)$ defines a $2$--form $\Phi\in\Omega^2(E)$ by 
\begin{equation}
\Phi(s_1,s_2)=g(s_1,J_E(s_2))\,,\,s_1,s_2\in\Gamma(E).
\label{VII3}
\end{equation}
Indeed, is easy to see that $\Phi(s_1,s_2)=-\Phi(s_2,s_1)$. Since $g$ is invariant with respect to $J_E$ we obtain that $\Phi$ is also invariant by $J_E$, namely $\Phi(J_E(s_1),J_E(s_2))=\Phi(s_1,s_2)$.

Since the Hermitian metric $g$ is a particular case of a Riemannian metric on an almost complex Lie algebroid $(E,\rho_E,[\cdot,\cdot]_E, J_E)$, we can consider the associated Levi-Civita connection.

Using the calculus on Lie algebroids,  by a similar argument as in geometry of almost Hermitian manifolds we have:
\begin{proposition}
\label{pIII1}
Let $(E,\rho_E,[\cdot,\cdot]_E, J_E,g)$ be an almost Hermitian Lie algebroid, $N_{J_E}$ the Nijenhuis tensor of $J_E$, $\Phi$ the $2$-form associated to $g$ and $D$ the Levi-Civita connection associated to $g$. Then
\begin{equation}
2g((D_{s_1}J_E)s_2,s_3)=d_E\Phi(s_1,J_E(s_2),J_E(s_3))-d_E\Phi(s_1,s_2,s_3)+g(N_{J_E}(s_2,s_3),J_E(s_1))
\label{VII5}
\end{equation}
for every $s_1,s_2,s_3\in\Gamma(E)$.
\end{proposition}

\begin{definition}
A linear connection $\nabla$ on an almost complex Lie algebroid $(E,\rho_E,[\cdot,\cdot]_E, J_E)$  is called \textit{almost complex connection} if the covariant derivative of $J_E$ with respect to $\nabla$ vanishes, that is
\begin{equation}
(\nabla_{s_1}J_E)s_2=\nabla_{s_1}(J_Es_2)-J_E(\nabla_{s_1}s_2)=0\,,\,\forall\,s_1,s_2\in\Gamma(E).
\label{V1}
\end{equation}
\end{definition}
The Proposition \ref{pIII1} says that the Levi-Civita connection of an almost Hermitian Lie algebroid is not an almost complex one. But, using also the Proposition \ref{pIII1} we can prove
\begin{theorem}
Let $(E,\rho_E,[\cdot,\cdot]_E, J_E,g)$ be an almost Hermitian Lie algebroid. The following conditions are equivalent:
\begin{enumerate}
\item[i)] the Levi-Civita connection is an almost complex connection;
\item[ii)] the almost complex structure $J_E$ is integrable and the associated $2$--form $\Phi$ is $d_E$--closed.
\end{enumerate}
\end{theorem}

\begin{remark}
When $J_E$ is integrable then $\Phi$ is $d_E$--closed if and only if the Levi-Civita connection is almost complex.
\end{remark}

\begin{definition}
An almost Hermitian Lie algebroid $(E,\rho_E,[\cdot,\cdot]_E, J_E,g)$ is said to be \textit{almost K\"{a}hlerian} if the fundamental $2$--form $\Phi$ is $d_E$--closed and, if moreover the almost complex structure $J_E$ is integrable then it is said to be \textit{K\"{a}hlerian}.
\end{definition}

\begin{example}
\label{e3.4}
Let us consider the Hermitian Lie algebroid $(\mathcal{L}^p(E),\rho_{\mathcal{L}^p(E)},[\cdot,\cdot]_{\mathcal{L}^p(E)}, J^c_E,g^c)$ from Example \ref{e3.1} associated to a given Hermitian Lie algebroid $(E,\rho_E,[\cdot,\cdot]_E, J_E,g)$. As usual for tangent bundle case, see \cite{Ya-I}, the complet lift of the Levi-Civita $D$ of $g$ is defined by $D^c_{s_1^c}s_2^c=(D_{s_1}s_2)^c$ and it is the Levi-Civita connection of $g^c$ on  $\mathcal{L}^p(E)$. Now it is easy to see that if $DJ_E=0$ then $D^cJ_E^c=0$ which says that $(E,\rho_E,[\cdot,\cdot]_E, J_E,g)$ is a K\"{a}hlerian Lie algebroid iff $(\mathcal{L}^p(E),\rho_{\mathcal{L}^p(E)},[\cdot,\cdot]_{\mathcal{L}^p(E)}, J^c_E,g^c)$ is also K\"{a}hlerian.
\end{example}

\begin{remark}
Since a Hermitian metric $g$ on an almost complex Lie algebroid is a particular case of a Riemannian metric it is nondegenerated. On the other hand the almost complex structure $J_E$ is also nondegenerated, hence the associated $2$--form $\Phi$ is nondegenerated on $E$. Thus, every almost complex Lie algebroid admits a structure of almost symplectic Lie algebroid \cite{I-M-D-M-P} and a nondegenerate Poisson structure \cite{Po2}. 
\end{remark}

Let us consider now a K\"{a}hlerian Lie algebroid $(E,\rho_E,[\cdot,\cdot]_E, J_E,g)$ over a smooth manifold $M$, $D$ the associated Levi-Civita connection and $R$ its curvature tensor.

The Riemann curvature tensor of the  K\"{a}hlerian Lie algebroid $(E,\rho_E,[\cdot,\cdot]_E, J_E,g)$ is usually defined  as a $4$--linear map $\mathcal{R}:\Gamma(E)\times\Gamma(E)\times\Gamma(E)\times\Gamma(E)\rightarrow C^\infty(M)$ given by
\begin{displaymath}
\mathcal{R}(s_1,s_2,s_3,s_4)=g(R(s_3,s_4)s_2,s_1).
\end{displaymath}
The standard properties of this map follows the classical ones from K\"{a}hlerian geometry. Also the sectional curvature of a K\"{a}hlerian Lie algebroid can be introduced as follows:

For every $x\in M$ and a given $2$--plane $P$ of $E_x$ ($2$-dimensional space in $E_x$) we define the function $K(P)$ by $K(P)=\mathcal{R}(s_1,s_2,s_1,s_2)$, where  $\{s_1,s_2\}$ is an orthonormal frame in $P$.

We will consider the restriction of $K$ to all $2$--planes from $E_x$ which are invariant by $J_E$, that is, for every section $s\in E_x$ we have $J_E(s)\in E_x$. 

\begin{definition}
For a given K\"{a}hlerian Lie algebroid $(E,\rho_E,[\cdot,\cdot]_E, J_E,g)$, the sectional curvature of $E$ in direction of the $J_E$ invariant $2$--plane $P$ is defined as usual by 
\begin{displaymath}
K(P)=\frac{\mathcal{R}(s_1,s_2,s_1,s_2)}{g(s_1,s_1)g(s_2,s_2)-g^2(s_1,s_2)}
\end{displaymath}
and  is called the \textit{holomorphic sectional curvature in direction of $P$} of $(E,\rho_E,[\cdot,\cdot]_E, J_E,g)$.  The restriction of sectional curvature to all $J_E$ invariant $2$-planes $P$, is called \textit{holomorphic sectional curvature} of $(E,\rho_E,[\cdot,\cdot]_E, J_E,g)$. 
\end{definition}

For a given $J_E$ invariant $2$--plane $P$, we can consider orthonormal frames of the form $\{s,J_E(s)\}$, where $s$ are unitary sections in $E_x$. Then the holomorphic sectional curvature is then given by
\begin{displaymath}
K(P)=\mathcal{R}(s,J_E(s),s,J_E(s))
\end{displaymath}
and the Riemann curvature tensor of a K\"{a}hlerian Lie algebroid $(E,\rho_E,[\cdot,\cdot]_E, J_E,g)$ is completely determined by the holomorphic sectional curvature defined for all $J_E$ invariant planes.

If the function  $K$ is constant for all $J_E$ invariant $2$--planes from $E_x$ and for every $x\in M$ then the K\"{a}hlerian Lie algebroid $(E,\rho_E,[\cdot,\cdot]_E, J_E,g)$ is said to be of \textit{constant holomorphic sectional curvature}.

We have the following Schur type theorem:
\begin{theorem}
\label{Schur}
Let $(E,\rho_E,[\cdot,\cdot]_E, J_E,g)$ be a K\"{a}hlerian transitive Lie algebroid with ${\rm rank\,}E=2m\geq 4$. If the holomorphic sectional curvature depends only of $x\in M$ (it is independent of the $J_E$-invariant $2$-plane $P$) then $(E,\rho_E,[\cdot,\cdot]_E, J_E,g)$ is of constant holomorphic sectional curvature.
\end{theorem}
\begin{proof}
Follows in the same manner as in Schur theorem for K\"{a}hler manifolds.
\end{proof}
\begin{remark}If the hypothesis of the theorem does not make the assumption that the anchor is surjective, then the real function given by the sectional curvature is constant on every leaf of the singular foliation of the Lie algebroid.
\end{remark}

\section{Hermitian metrics on the associated complex Lie algebroid and some related structures}
\setcounter{equation}{0}
In this section the Hermitian metrics and linear connections compatible with such metrics on the associated complex Lie algebroid are studied and we present the Levi-Civita connection associated to a such metric. Also, we describe some $E$-Chern forms of $E^{1,0}$ associated to an almost complex connection $\nabla$ on $E$. Finally, we consider a metric product connection associated to an almost Hermitian Lie algebroid and a $2$--form section for $E^{0,1}$ similar to the second fundamental form of complex distributions is studied in our setting.

\subsection{Linear connections compatible with Hermitian metrics  on the associated complex Lie algebroid $(E_{\mathbb{C}},\rho_E,[\cdot,\cdot]_E)$}
A linear connection $\nabla$ on the almost complex Lie algebroid $(E,\rho_E,[\cdot,\cdot]_E, J_E)$ over a smooth manifold $M$ extend by $\mathbb{C}$-linearity to a connection on the complexified Lie algebroid $(E_{\mathbb{C}},\rho_E,[\cdot,\cdot]_E)$, \cite{W}, as follows:
\begin{enumerate}
\item[(1)] $\nabla$ is $\mathbb{C}$--bilinear;
\item[(2)] $\nabla_{fs_1}s_2=f\nabla_{s_1}s_2$, for all $f\in C^\infty(M)_{\mathbb{C}}$ and $s_1,s_2\in\Gamma(E_{\mathbb{C}})$;
\item[(3)] $\nabla_{s_1}(fs_2)=(\rho_E(s_1)f)s_2+f\nabla_{s_1}s_2$, for all $f\in C^\infty(M)_{\mathbb{C}}$ and $s_1,s_2\in\Gamma(E_{\mathbb{C}})$.
\end{enumerate}

If $\{e_a\}$, $a=1,\ldots,m$ is a local basis of $E^{1,0}$ and $\{\overline{e_a}\}$, $a=1,\ldots,m$ is a local basis of $E^{0,1}$, then a linear connection $\nabla$ on the complex Lie algebroid $(E_{\mathbb{C}},\rho_E,[\cdot,\cdot]_E)$ is locally given by
\begin{displaymath}
\nabla_{e_a}e_b=\Gamma^c_{ab}e_c+\Gamma^{\overline{c}}_{ab}\overline{e_c}\,,\,\nabla_{e_a}\overline{e_b}=\Gamma^c_{a\overline{b}}e_c+\Gamma^{\overline{c}}_{a\overline{b}}\overline{e_c},
\end{displaymath}
\begin{displaymath}
\nabla_{\overline{e_a}}e_b=\Gamma^c_{\overline{a}b}e_c+\Gamma^{\overline{c}}_{\overline{a}b}\overline{e_c}\,,\,\nabla_{\overline{e_a}}\overline{e_b}=\Gamma^c_{\overline{a}\,\overline{b}}e_c+\Gamma^{\overline{c}}_{\overline{a}\,\overline{b}}\overline{e_c}.
\end{displaymath} 
By $\mathbb{C}$--linearity condition it follows that $\nabla$ satisfies  $\overline{\nabla_{s_1}s_2}=\nabla_{\overline{s_1}}\overline{s_2}$, hence we get
\begin{displaymath}
\Gamma^c_{ab}=\overline{\Gamma^{\overline{c}}_{\overline{a}\,\overline{b}}}\,,\,\Gamma^c_{a\overline{b}}=\overline{\Gamma^{\overline{c}}_{\overline{a}b}}\,,\,\Gamma^c_{\overline{a}b}=\overline{\Gamma^{\overline{c}}_{a\overline{b}}}\,,\,\Gamma^{\overline{c}}_{ab}=\overline{\Gamma^c_{\overline{a}\,\overline{b}}}.
\end{displaymath} 
On the other hand we have $J_E(e_a)=ie_a$ and $J_E(\overline{e_b})=i\overline{e_b}$, thus $\nabla$ is almost complex if and only if
$\Gamma^{\overline{c}}_{ab}=\Gamma^{\overline{c}}_{\overline{a}b}=\Gamma^{c}_{a\overline{b}}=\Gamma^{c}_{\overline{a}\,\overline{b}}=0$. Indeed, from $\nabla_{e_a}J_E(e_b)=J_E(\nabla_{e_b}e_b)$ we obtain $\Gamma^{\overline{c}}_{ab}=0$. Similarly, we get the vanishing of the others coefficients. Thus, in this case we have 
\begin{displaymath}
\nabla_{e_a}e_b=\Gamma^c_{ab}e_c\,,\,\nabla_{e_a}\overline{e_b}=\Gamma^{\overline{c}}_{a\overline{b}}\overline{e_c}\,,\,\nabla_{\overline{e_a}}e_b=\Gamma^c_{\overline{a}b}e_c\,,\,\nabla_{\overline{e_a}}\overline{e_b}=\Gamma^{\overline{c}}_{\overline{a}\,\overline{b}}\overline{e_c},
\end{displaymath}
which says that an almost complex connection on  $E_{\mathbb{C}}$ preserves the distributions of $J_E$.

If we consider a Hermitian metric $g$ on the almost complex Lie algebroid $(E,\rho_E,[\cdot,\cdot]_E, J_E)$ over a smooth manifold $M$, then we can naturally extend it to a Hermitian metric on the complex Lie algebroid $(E_{\mathbb{C}},\rho_E,[\cdot,\cdot]_E)$. 
\begin{remark}
If $\Phi$ is the associated fundamental $2$--form, as usual we can prove that $\Phi\in\Omega^{1,1}(E)$.
\end{remark}
\begin{remark}
Let $\{e_1,\ldots,e_m\}$ a  basis of $E^{1,0}$ over $\mathbb{C}$ and $\{e^1,\ldots,e^m\}$ its dual basis on $(E^{1,0})^*$ such that $\{\overline{e_1},\ldots,\overline{e_m}\}$ is a basis of $E^{0,1}$ over $\mathbb{C}$ and $\{\overline{e^1},\ldots,\overline{e^m}\}$ is the dual basis on $(E^{0,1})^*$. If we denote by $g_{a\overline{b}}=g(e_a,\overline{e_b})$, $a,b=1,\ldots,m$ then
$g_{a\overline{b}}=\overline{g_{b\overline{a}}}$ and $\Phi=-i\sum\limits_{a,b=1}^mg_{a\overline{b}}e^a\wedge\overline{e^b}$,
where the  exterior product is that in Lie algebroids framework, see \cite{M}.
\end{remark}

Similarly to Riemannian Lie algebroids case, we have
\begin{definition}
A linear connection on the complex Lie algebroid $(E_{\mathbb{C}},\rho_E,[\cdot,\cdot]_E)$ over $M$ is said to be \textit{compatible} with the Hermitian metric $g$ on $E_{\mathbb{C}}$ if $\nabla_sg=0$, $\forall\,s\in\Gamma(E_{\mathbb{C}})$, that is, for every two sections $s_1,s_2\in\Gamma(E_{\mathbb{C}})$ we have
\begin{equation}
\label{II.3.5}
(\nabla_sg)(s_1,s_2)=\rho_E(s)(g(s_1,s_2))-g(\nabla_ss_{1},s_2)-g(s_1,\nabla_ss_2)=0.
\end{equation}
\end{definition}

\begin{proposition}
The curvature $R$ of a linear connection $\nabla$ on $(E_{\mathbb{C}},\rho_E,[\cdot,\cdot]_E)$ over $M$, which is compatible with the Hermitian metric $g$ on $E_{\mathbb{C}}$, satisfies
\begin{displaymath}
g(R(s_1,s_2)s,s^{'})=-\overline{g(R(s_1,s_2)s^{'},s)}\,,\,\forall\,s,s^{'}s_1,s_2\in\Gamma(E_{\mathbb{C}}).
\end{displaymath} 
\end{proposition}
\begin{proof}
Follows by direct calculation, using the relation \eqref{II.3.5}.
\end{proof}
\begin{remark}
\label{rcurv}
If $\{e_a\}$, $a=1,\ldots,m$ is a local orthonormal frame on $E_{\mathbb{C}}$, the matrix of curvature $2$--forms $(R^b_a)$ with respect to this frame has the property $R^b_a(s_1,s_2) =-\overline{R^a_b(s_1,s_2)}$ or $R=-\overline{R^t}$ in matrix notation. 
\end{remark}

In the following of this subsection we describe the Levi-Civita connection on complex Lie algebroid $(E_{\mathbb{C}},\rho_E,[\cdot,\cdot]_E,g)$ over $M$, when $(E,\rho_E,[\cdot,\cdot]_E,J_E,g)$ is almost Hermitian or K\"{a}hlerian. Let us put $g_{a\overline{b}}=g(e_a,\overline{e_b})$ and $g(e_a,e_b)=g(\overline{e_a},\overline{e_b})=0$. Then, we have
\begin{proposition}
\label{coef}
Let $(E,\rho_E,[\cdot,\cdot]_E,J_E,g)$ be an almost Hermitian Lie algebroids over $M$. Then the local coefficients of the Levi-Civita connection on the  associated Hermitian complex Lie algebroid $(E_{\mathbb{C}},\rho_E,[\cdot,\cdot]_E,g)$ are given by
\begin{eqnarray*}
\Gamma^{d}_{ab}=\frac{1}{2}g^{\overline{c}d}\left(\rho_E(e_a)( g_{b\overline{c}})+\rho_E(e_b)(g_{a\overline{c}})+C^{e}_{ab}g_{e\overline{c}}-C^{\overline{e}}_{b\overline{c}}g_{a\overline{e}}+C^{\overline{e}}_{\overline{c}a}g_{b\overline{e}}\right),
\end{eqnarray*}
\begin{eqnarray*}
\Gamma^{d}_{a\overline{b}}=\frac{1}{2}g^{\overline{c}d}\left(\rho_E(e_{\overline{b}})(g_{a\overline{c}})-\rho_E(e_{\overline{c}})(g_{a\overline{b}})+C^{e}_{a\overline{b}}g_{e\overline{c}}-C^{\overline{e}}_{\overline{b}\,\overline{c}}g_{a\overline{e}}+C^{e}_{\overline{c}a}g_{e\overline{b}}\right),
\end{eqnarray*}
\begin{eqnarray*}
\Gamma^{d}_{\overline{a}b}=\frac{1}{2}g^{\overline{c}d}\left(\rho_E(e_{\overline{a}})(g_{b\overline{c}})-\rho_E(e_{\overline{c}})(g_{b\overline{a}})+C^{e}_{\overline{a}b}g_{e\overline{c}}-C^{e}_{b\overline{c}}g_{e\overline{a}}+C^{\overline{e}}_{\overline{c}\,\overline{a}}g_{b\overline{e}}\right),
\end{eqnarray*}
\begin{eqnarray*}
\Gamma^{\overline{d}}_{ab}=\frac{1}{2}g^{\overline{d}c}\left(C^{\overline{e}}_{ab}g_{c\overline{e}}-C^{\overline{e}}_{bc}g_{a\overline{e}}+C^{\overline{e}}_{ca}g_{b\overline{e}}\right),
\end{eqnarray*}
and their conjugates, where $(g^{\overline{b}a})$ is the inverse matrix of $(g_{a\overline{b}})$. 
\end{proposition}
\begin{proof}
Follows by direct calculus, taking in the formula \eqref{III4} which gives the Levi-Civita connection, the following combinations: $s_1=e_a,s_2=e_b,s_3=\overline{e_c}$, $s_1=e_a,s_2=\overline{e_b},s_3=\overline{e_c}$, $s_1=\overline{e_a},s_2=e_b,s_3=\overline{e_c}$ and $s_1=e_a,s_2=e_b,s_3=e_c$, respectively.
\end{proof}

\begin{corollary}
If the almost complex Lie algebroid $(E,\rho_E,[\cdot,\cdot]_E,J_E,g)$ over $M$ is K\"{a}hlerian then the possible nonzero coefficients of the Levi-Civita connection on the associated complex Lie algebroid $(E_{\mathbb{C}},\rho_E,[\cdot,\cdot]_E,g)$ are
\begin{displaymath}
\Gamma^d_{ab}=g^{\overline{c}d}\left(\rho_E(e_a)(g_{b\overline{c}})+C^{\overline{e}}_{\overline{c}a}g_{b\overline{e}}\right)\,,\,\Gamma^d_{\overline{a}b}=C^d_{\overline{a}b}.
\end{displaymath}
\end{corollary}
\begin{proof}
Since $(E,\rho_E,[\cdot,\cdot]_E,g)$ is K\"{a}hlerian, we have that $J_E$ is integrable and $d_E\Phi=0$. Then from Newlander-Nirenberg theorem we have $C^{\overline{c}}_{ab}=0$ and then $\Gamma^{\overline{d}}_{ab}=0$. Also, we have
\begin{eqnarray*}
0=id_E\Phi&=&\left(\rho_E(e_c)(g_{a\overline{b}})-\rho_E(e_a)(g_{c\overline{b}})-C^d_{ca}g_{d\overline{b}}-C^{\overline{d}}_{c\overline{b}}g_{a\overline{d}}+C^{\overline{d}}_{a\overline{b}}g_{c\overline{d}}\right)e^c\wedge e^{a}\wedge \overline{e^b}\\
&&+\left(\rho_E(e_{\overline{c}})(g_{a\overline{b}})-\rho_E(e_{\overline{b}})(g_{a\overline{c}})-C^d_{\overline{c}a}g_{d\overline{b}}-C^{\overline{d}}_{\overline{c}\,\overline{b}}g_{a\overline{d}}-C^{d}_{a\overline{b}}g_{d\overline{c}}\right)\overline{e^c}\wedge e^{a}\wedge \overline{e^b}
\end{eqnarray*} 
which implies
\begin{displaymath}
\rho_E(e_b)(g_{a\overline{c}})=\rho_E(e_a)(g_{b\overline{c}})+C^d_{ba}g_{d\overline{c}}+C^{\overline{d}}_{b\overline{c}}g_{a\overline{d}}-C^{\overline{d}}_{a\overline{c}}g_{b\overline{d}},
\end{displaymath}
\begin{displaymath}
\rho_E(e_{\overline{b}})(g_{a\overline{c}})=\rho_E(e_{\overline{c}})(g_{a\overline{b}})-C^d_{\overline{c}a}g_{d\overline{b}}-C^{\overline{d}}_{\overline{c}\,\overline{b}}g_{a\overline{d}}-C^{d}_{a\overline{b}}g_{d\overline{c}}.
\end{displaymath}
Now, replacing the above relations in the expression of coefficients of Levi-Civita connection we get
\begin{displaymath}
\Gamma^d_{ab}=g^{\overline{c}d}\left(\rho_E(e_a)(g_{b\overline{c}})+C^{\overline{e}}_{\overline{c}a}g_{b\overline{e}}\right)\,,\,\Gamma^d_{a\overline{b}}=0\,,\,\Gamma^d_{\overline{a}b}=C^d_{\overline{a}b}.
\end{displaymath}
which ends the proof.
\end{proof}
Also, by direct calculus, we have
\begin{proposition}
If the almost complex Lie algebroid $(E,\rho_E,[\cdot,\cdot]_E,J_E,g)$ over $M$ is K\"{a}hlerian then the nonzero curvatures of the Levi-Civita connection on the associated complex Lie algebroid $(E_{\mathbb{C}},\rho_E,[\cdot,\cdot]_E,g)$ are
\begin{eqnarray*}
&&R^d_{ab,c}=\rho_E(e_a)(\Gamma^d_{bc})-\rho_E(e_b)(\Gamma^d_{ac})+\Gamma^{e}_{bc}\Gamma^d_{ae}-\Gamma^{e}_{ac}\Gamma^d_{be}-C^{e}_{ab}\Gamma^d_{ec},\\
&&R^{\overline{d}}_{a\overline{b},\overline{c}}=\rho_E(e_{a})(\Gamma^{\overline{d}}_{\overline{b\,\overline{c}}})-C^{\overline{e}}_{a\overline{b}}\Gamma^{\overline{d}}_{\overline{e}\,\overline{c}},\,R^{\overline{d}}_{\overline{a}\,\overline{b},\overline{c}}=\overline{R^d_{ab,c}}\,,\,R^d_{a\overline{b},c}=-\overline{R^{\overline{d}}_{b\overline{a},\overline{c}}},
\end{eqnarray*}
where we have put $R(e_a,e_b)e_c=R^d_{ab,c}e_d$ and similarly for the other components of curvature.
\end{proposition}

\begin{remark}
In \cite{L-T-W} the (complex) para-K\"{a}hlerian Lie algebroid is defined as a (complex) symplectic Lie algebroid $(E,\omega)$ with a splitting $E=E^{1,0}\oplus E^{0,1}$ as the direct sum a two polarizations, where a polarization of $(E,\omega)$ means a lagrangian subalgebroid of $E$, i.e. a subbundle which is closed under brackets and maximal isotropic with respect to $\omega$. Then for a para-K\"{a}hlerian Lie algebroid is proved that there is a unique torsion-free $E$--connection $\nabla$ on $E$ (called para-K\"{a}hlerian connection), for which covariant differentiation leaves the para-K\"{a}hler structure invariant, i.e. for every $s_1,s_2,s_3\in\Gamma(E)$, $\nabla_{s_1}$ leaves the splitting invariant, and $\rho_E(s_1)(\omega(s_2,s_3))=\omega(\nabla_{s_1}s_2,s_3)+\omega(s_2,\nabla_{s_1}s_3)$. The curvature of this connection is in $\left((E^{1,0})^*\wedge(E^{0,1})^*\right)\otimes End(E)$.
\end{remark}

\subsection{Chern forms of $E^{1,0}$ associated to an almost complex connection}

There exists a general construction of characteristic classes of a Lie algebroid $E$ \cite{Fe},  which mimics the Chern-Weil theory. More exactly, if we take a vector bundle $F\rightarrow M$ of rank $r$ and $\nabla:\Gamma(E)\times\Gamma(F)\rightarrow\Gamma(F)$ an $E$-connection on $F$, as well as we seen, the curvature operator of $\nabla$ denoted by $R_\nabla$ is a $F\otimes F^*$-valued $E$-form section of degree $2$. Let us consider $I^k(Gl(r,\mathbb{R}))$ the space of real, $ad$-invariant, symmetric, $k$-multilinear functions on the Lie algebra $gl(r,\mathbb{R})$. Then for every $\phi\in I^k(Gl(r,\mathbb{R}))$ the $2k$-forms on $E$ given by $\phi(R_\nabla)=\phi(R_\nabla,\ldots,R_\nabla)\in\Omega^{2k}(E)$ are $d_E$-closed, called the \textit{$E$-Chern forms of order $k$ of $F$}, and their define the $E$-cohomology classes $[\phi(R_\nabla)]\in H^{2k}(E)$ which are the \textit{$E$-principal characteristic classes of order $k$} of $F$, called also the \textit{$E$-Chern classes of $F$}. The characteristic classes are spanned by the classes $[c_k(R_\nabla)]$, where $c_k(R_\nabla)$ is the sum of the principal minors of order $k$ in $\det(R_\nabla-\lambda I)$, and do not depend on the choice of the connection $\nabla$.

In this subsection, following some ideas from \cite{N-O}, we use a real representation of an almost complex Lie algebroid $(E,\rho_E,[\cdot,\cdot]_E,J_E)$ to get some real $d_E$-closed forms on $E$ such that their $E$-cohomology classes generate the $E$-Chern ring of the complex vector bundle $E^{1,0}$. These forms are obtained using the almost complex structure $J_E$ on $E$ and the curvature forms of a (real) almost complex connection on $E$. These forms are called \textit{$E$-Chern forms} of $E^{1,0}$ associated to $\nabla$. 

Let $\nabla$ be an almost complex connection on the almost complex Lie algebroid $(E,\rho_E,[\cdot,\cdot]_E,J_E)$. Then as well as we seen $\nabla$ can be usually extended to a connection $\nabla^{\mathbb{C}}$ on $E_{\mathbb{C}}$ and the property $\nabla J_E=0$ of $\nabla$ implies that the covariant derivatives of the sections of $E^{1,0}$ are also section in $E^{1,0}$. We also consider $\{e_a,J_E(e_a)\}$, $a=1,\ldots,m$ be a local frame for the sections of $E$ over $U\subset M$, and for a better presentation of the notions in this subsection we shall denote $J_E(e_a)=e_{a^*}$, $a=1\ldots,m$. Then $\{e^{1,0}_a=e_a-ie_{a^*}\}$, $a=1\ldots,m$ is a local frame of $E^{1,0}\subset E_{\mathbb{C}}$. 

Let $R$ be the curvature tensor section of $\nabla$ and let $R^{\mathbb{C}}$ be the curvature tensor section of $\nabla^{\mathbb{C}}$. Then, we easily get that 
\begin{displaymath}
R^{\mathbb{C}}(s_1,s_2)e^{1,0}_a=R(s_1,s_2)e_a-iR(s_1,s_2)e_{a^*},
\end{displaymath}
for every real sections $s_1,s_2\in\Gamma(E)$.

Since $\nabla J_E=0$ we easily have $J_ER(s_1,s_2)=R(s_1,s_2)J_E$  and let $\left( 
\begin{array}{cccc}
R^{a}_b & -R^{a^*}_b  \\ 
R^{a^*}_b & R^{a}_b 
\end{array}
\right)$, $a,b=1\ldots,m$, be the matrix of $J_ER$ with respect to the local basis $\{e_a,e_{a^*}\}$, $a=1\ldots,m$, i.e.
\begin{eqnarray*}
J_ER(s_1,s_2)e_a&=&\sum_{b=1}^m(R^b_a(s_1,s_2)e_b+R^{b^*}_a(s_1,s_2)e_{b^*}),
\end{eqnarray*}
\begin{eqnarray*}
(J_ER)(s_1,s_2)e_{a^*}&=&(J_ER)(s_1,s_2)J_E(e_a)=J_E(J_ER)(s_1,s_2)e_a\\
&=&\sum_{b=1}^m(-R^{b^*}_a(s_1,s_2)e_b+R^{b}_a(s_1,s_2)e_{b^*}).
\end{eqnarray*}
The curvature matrix $\Phi$ of the restriction of $\nabla^{\mathbb{C}}$ to $E^{1,0}\subset E_{\mathbb{C}}$, with  respect to the local basis $\{e^{1,0}_a\}$, $a=1,\ldots,m$ is given by
\begin{eqnarray*}
\sum_{b=1}^m\Phi_a^b(s_1,s_2)e^{1,0}_b&=&R^{\mathbb{C}}(s_1,s_2)e^{1,0}_a=R(s_1,s_2)e_a-iR(s_1,s_2)e_{a^*}\\
&=&-(J_ER)(s_1,s_2)e_{a^*}-i(J_ER)(s_1,s_2)e_a\\
&=&\sum_{b=1}^m(R^{b^*}_a(s_1,s_2)e_b-R^b_a(s_1,s_2)e_{b^*}-iR^b_a(s_1,s_2)e_b-iR^{b^*}_a(s_1,s_2)e_{b^*})\\
&=&\sum_{b=1}^m(R^{b^*}_a(s_1,s_2)(e_b-ie_{b^*})-iR^b_a(s_1,s_2)(e_b-ie_{b^*}))\\
&=&\sum_{b=1}^m(R^{b^*}_a(s_1,s_2)-iR^b_a(s_1,s_2))e^{1,0}_b.
\end{eqnarray*}
Thus, we have
\begin{equation}
\label{B1}
\Phi^b_a=R^{b^*}_a-iR^{b}_a=-i(R^b_a+iR^{b^*}_a).
\end{equation}
Suppose now, that $\nabla$ is compatible with a Hermitian metric $g$ on $E$. If $\{e_a,e_{a^*}\}$, $a=1\ldots,m$ is an orthonormal local basis of sections of $E$ with respect to this metric then the matrix $\Phi$ has the property $\Phi^t=-\overline{\Phi}$, see Remark \ref{rcurv}.

Thus, the $E$-Chern forms of $E^{1,0}$ are essentially determined by the matrix
\begin{equation}
\label{B2}
i\Phi=R+iR^*
\end{equation}
and every element $\phi\in I^k(Gl(m,\mathbb{C}))$ obtained from $\det(A-\lambda I)$, where $A\in gl(m,\mathbb{C})$.

We denoted by $R$ the $m\times m$ matrix with entries the $2$-forms $R^b_a$ and $R^*$ the $m\times m$ matrix with entries the $2$-forms $R^{b^*}_a$. On the other hand it is well known that the elements of $I^k(Gl(m,\mathbb{C}))$ are also generated by the polynomials
\begin{displaymath}
{\rm trace}\,A^k\,,\,k=0,1,\ldots,m\,,\,A\in gl(m,\mathbb{C}).
\end{displaymath}
These polynomials can be used to construct some $d_E$-closed forms on $E$ which are called \textit{$E$-Chern forms of $E^{1,0}$ associated to $\nabla$}.
\begin{theorem}
\label{Chern}
The $E$-Chern forms of $E^{1,0}$ given by ${\rm trace}\,(i\Phi)^k$ are, up to a constant factor the forms on $E$, given by ${\rm trace}\,\left( 
\begin{array}{cccc}
R & -R^{*}  \\ 
R^{*} & R 
\end{array}
\right)^k$, i.e. they can be constructed using the matrix of $J_ER$, where $J_E$ is the almost complex structure on $E$ and $R$ is the curvature of an almost complex connection $\nabla$ on $E$.
\end{theorem}
\begin{proof}
We use an orthonormal local basis $\{e_a,e_{a^*}\}$, $a=1,\ldots,m$ with respect to a Hermitian metric on $E$. Then we have
\begin{displaymath}
R^t+iR^{*t}=(i\Phi)^t=\overline{i\Phi}=R-iR^*.
\end{displaymath}
Thus, $R$ is a symmetric matrix and $R^*$ is a skew-symmetric matrix. The form ${\rm trace}\,(i\Phi)^k$ is real since we have
\begin{eqnarray*}
\overline{{\rm trace}\,(i\Phi)^k}&=&{\rm trace}\,\overline{(i\Phi)^k}={\rm trace}\,((i\Phi)^t)^k={\rm trace}\,((i\Phi)^k)^t={\rm trace}\,(i\Phi)^k.
\end{eqnarray*} 
It is well known that to the complex matrix $i\Phi=R+iR^*$ it corresponds its real representation $\left( 
\begin{array}{cccc}
R & -R^{*}  \\ 
R^{*} & R 
\end{array}
\right)$, and to $(i\Phi)^k=\Pi+i\Pi^*$ corresponds $\left( 
\begin{array}{cccc}
\Pi & -\Pi^{*}  \\ 
\Pi^{*} & \Pi 
\end{array}
\right)=\left( 
\begin{array}{cccc}
R & -R^{*}  \\ 
R^{*} & R 
\end{array}
\right)^k$.

Then we have
\begin{eqnarray*}
{\rm trace}\,(i\Phi)^k&=&{\rm trace}\,(\Pi+i\Pi^*)={\rm trace}\,\Pi=\frac{1}{2}{\rm trace}\,\left( 
\begin{array}{cccc}
\Pi & -\Pi^{*}  \\ 
\Pi^{*} & \Pi 
\end{array}
\right)=\frac{1}{2}{\rm trace}\,\left( 
\begin{array}{cccc}
R & -R^{*}  \\ 
R^{*} & R 
\end{array}
\right)^k.
\end{eqnarray*}
Since the $E$-Chern forms of $E^{1,0}$ do not depend on the local frame used for their local representation, it follows that we can use the matrices associated to $J_ER$ to get them.
\end{proof}

The result obtained above can be extended to the case of an almost complex connection $\nabla$ on $E$ which is not necessarily compatible with a Hermitian metric on $E$. In this case the imaginary part of the form ${\rm trace}\,(i\Phi)^k$ is a $d_E$-exact form on $E$ and the corresponding $E$-Chern form of order $k$ of $E^{1,0}$ is the real part of $(i\Phi)^k$. Then the $E$-Chern form of order $k$ of $E^{1,0}$ given by trace of the $k$-power of the matrix associated with $J_ER$ in an arbitrary local frame of $E$.

\subsection{A metric product connection}
In this subsection we consider a metric product connection associated to an almost Hermitian Lie algebroid, following some arguments from the theory of complex distributions, see \cite{Va-Iasi}. 
Also, a $2$--form section for $E^{0,1}$ similar to the second fundamental form of complex distributions is studied in our setting.

Let $E_{\mathbb{C}}=E^{1,0}\oplus E^{0,1}$ be the  complexification of an almost Hermitian Lie algebroid. We consider the natural projections $p^{1,0}:E_{\mathbb{C}}\rightarrow E^{1,0}$ and $p^{0,1}:E_{\mathbb{C}}\rightarrow E^{0,1}$  defined in \eqref{m1}.

It is easy to see that $E^{1,0}$ is the $h$-orthogonal of $E^{0,1}$ in $E_{\mathbb{C}}$ and conversely, where $h$ is the Hermitian metric on $E_{\mathbb{C}}$ defined by $h(s_1,s_2)=g(s_1,\overline{s_2})$, for every $s_1,s_2\in\Gamma(E_{\mathbb{C}})$. We define the \textit{metric product connection} $\widetilde{D}$ by
\begin{equation}
\label{m2}
\widetilde{D}_{s_1}s_2=p^{0,1}D_{s_1}p^{0,1}s_2+p^{1,0}D_{s_1}p^{1,0}s_2=D_{s_1}s_2+\frac{1}{2}\left(D_{s_1}J_E\right)J_E(s_2),\,\forall\,s_1,s_2\in\Gamma(E_{\mathbb{C}}),
\end{equation}
where $D$ is the Levi-Civita connection on $E_{\mathbb{C}}$. We notice that $\widetilde{D}$ satisfies the conditions
\begin{equation}
\label{m3}
\widetilde{D}p^{1,0}=\widetilde{D}p^{0,1}=\widetilde{D}h=0
\end{equation}
and has the torsion
\begin{equation}
\label{m4}
T_{\widetilde{D}}(s_1,s_2)=p^{0,1}\left(D_{s_2}p^{1,0}s_1-D_{s_1}p^{1,0}s_2\right)+p^{1,0}\left(D_{s_2}p^{0,1}s_1-D_{s_1}p^{0,1}s_2\right)
\end{equation}
or, by using \eqref{m1}, we get
\begin{equation}
\label{m5}
T_{\widetilde{D}}(s_1,s_2)=\frac{1}{2}\left(\left(D_{s_1}J_E\right)J_E(s_2)-\left(D_{s_2}J_E\right)J_E(s_1)\right).
\end{equation}
If $\{e_a\}$, $a=1,\ldots,m$ is a local basis of $E^{1,0}$ and $\{\overline{e_a}\}$, $a=1,\ldots,m$ is a local basis of $E^{0,1}$then, we have
\begin{displaymath}
\widetilde{D}_{e_a}e_b=\Gamma^d_{ab}e_d\,,\,\widetilde{D}_{e_a}\overline{e_b}=\Gamma^{\overline{d}}_{a\overline{b}}\overline{e_d}\,,\,\widetilde{D}_{\overline{e_a}}e_b=\Gamma^d_{\overline{a}b}e_d\,,\,\widetilde{D}_{\overline{e_a}}\overline{e_b}=\Gamma^{\overline{d}}_{\overline{a}\,\overline{b}}\overline{e_d}
\end{displaymath}
and
\begin{displaymath}
T_{\widetilde{D}}(e_a,e_b)=C^{\overline{d}}_{ba}\overline{e_d}\,,\,T_{\widetilde{D}}(e_a,\overline{e_b})=(C^d_{\overline{b}a}-\Gamma^d_{\overline{b}a})e_d+(\Gamma^{\overline{d}}_{a\overline{b}}-C^{\overline{d}}_{a\overline{b}})\overline{e_d},
\end{displaymath}
\begin{displaymath}T_{\widetilde{D}}(\overline{e_a},e_b)=-T_{\widetilde{D}}(e_b,\overline{e_a})=\overline{T_{\widetilde{D}}(\overline{e_a},e_b)}\,,\,T_{\widetilde{D}}(\overline{e_a},\overline{e_b})=\overline{T_{\widetilde{D}}(e_a,e_b)},
\end{displaymath}
where $\Gamma^\cdot_{\cdot\cdot}$ are the coefficients of the Levi-Civita connection from Proposition \ref{coef}.

As in the classical theory of submanifolds, the tensor section
\begin{equation}
\label{m6}
B^{0,1}(s_1,s_2)=-\frac{1}{2}\left(D_{p^{0,1}s_1}J_E\right)J_E(p^{0,1}s_2)=p^{1,0}\left(D_{p^{0,1}s_1}p^{0,1}s_2\right)
\end{equation} 
will be called the \textit{second fundamental form section} of $E^{0,1}$ and the equations:
\begin{equation}
\label{m7}
D_{p^{0,1}s_1}p^{0,1}s_2=\widetilde{D}_{p^{0,1}s_1}p^{0,1}s_2-\frac{1}{2}\left(D_{p^{0,1}s_1}J_E\right)J_E(p^{0,1}s_2),
\end{equation}
\begin{equation}
\label{m8}
D_{p^{0,1}s_1}p^{1,0}s_2=\widetilde{D}_{p^{0,1}s_1}p^{1,0}s_2-\frac{1}{2}\left(D_{p^{0,1}s_1}J_E\right)J_E(p^{1,0}s_2)
\end{equation}
will be called the \textit{Gauss-Weingarten equations} of $E^{0,1}$.

Locally, $B^{0,1}$ is given by $B^{0,1}(e_a,e_b)=B^{0,1}(e_a,\overline{e_b})=B^{0,1}(\overline{e_a},e_b)=0\,,\,B^{0,1}(\overline{e_a},\overline{e_b})=\Gamma^d_{\overline{a}\,\overline{b}}e_d$.
\begin{remark}
If $(E,\rho_E,[\cdot,\cdot]_E,g)$ is Hermitian then the integrability of $J_E$ implies $C^{\overline{d}}_{ab}=0$ and taking into account the expression of the coefficients $\Gamma^{\overline{d}}_{ab}$ from Proposition \ref{coef} it follows that $B^{0,1}$ vanishes. In the end of this subsection we obtain that conversely is also true. 
\end{remark}
Similarly, the operators $W^{0,1}_s\in End_{\mathbb{C}}(E_{\mathbb{C}})$ defined by
\begin{equation}
\label{m10}
W^{0,1}_{s_2}s_1:=\frac{1}{2}\left(D_{p^{0,1}s_1}J_E\right)J_E(p^{1,0}s_2)=-p^{0,1}D_{p^{0,1}s_1}p^{1,0}s_2\,,\,\forall\,s_1,s_2\in\Gamma(E_{\mathbb{C}})
\end{equation}
will be called the \textit{Weingarten operators} for $E^{0,1}$ and the metric character of $\widetilde{D}$ implies
\begin{equation}
\label{m11}
h(W^{0,1}_{s_3}s_1,s_2)=h(s_3,B^{0,1}(s_1,s_2)).
\end{equation}
Locally, we have $W^{0,1}_{e_b}e_a=W^{0,1}_{\overline{e_b}}e_a=W^{0,1}_{\overline{e_b}}\overline{e_a}=0\,,\,W^{0,1}_{e_b}\overline{e_a}=-\Gamma^{\overline{d}}_{\overline{a}b}\overline{e_d}$.

The relation \eqref{m5} says that the second fundamental form section $B^{0,1}$ of $E^{0,1}$ is symmetric iff $E^{0,1}$ is integrable, and, in this case, $\widetilde{D}|_{E^{0,1}}$ may be seen as a metric torsionless connection along $E^{0,1}$. Of course, as usual, we can relate the curvature tensors of the connections $D$ and $\widetilde{D}$ and then the analogue formulas for Gauss, Ricci and Codazzi equations can be obtained for our case.

Also, the second fundamental form section $B^{0,1}$ provides a \textit{mean curvature section} $H^{0,1}$ of $E^{0,1}$, given by
\begin{equation}
\label{m13}
H^{0,1}=H(E^{0,1})={\rm tr}\,B^{0,1}=\sum_{a,b=1}^mB^{0,1}(f_a,\overline{f_b})\in\Gamma(E^{1,0}),
\end{equation}
where $\{f_a\}$, $a=1,\ldots,m$ is an arbitrary, $h$-unitary local basis of $E_{\mathbb{C}}$.

The relation between $H^{0,1}$ and Weingarten operators is provided by the $h$-dual $1$--form section $k^{0,1}$ of $H^{0,1}$ which is given by
\begin{equation}
\label{m14}
k^{0,1}(s)=\sum_{a=1}^mh(W^{0,1}_sf_a,\overline{f_a}).
\end{equation}
As usual, will we say that $E^{0,1}$ is \textit{totally geodesic} if its second fundamental form section $B^{0,1}$ vanishes and it is \textit{minimal} if $H^{0,1}=0$. If there exists a section $s\in\Gamma(E^{1,0})$ such that
\begin{equation}
\label{m15}
B^{0,1}(s_1,s_2)=g(p^{0,1}s_1,p^{0,1}s_2)s,
\end{equation}
the subbundle $E^{0,1}$ is \textit{totally umbilical}.
\begin{remark}
Taking into account that $g(p^{0,1}s_1,p^{0,1}s_2)=0$ it follows that $E^{0,1}$ is totally geodesic iff it is totally umbilical and in this case it is also minimal.
\end{remark}

Similar calculations as in the proof of Proposition 2.1 from \cite{Va-Iasi} leads to 
\begin{proposition}
Let $(E,\rho_E,[\cdot,\cdot]_E,g)$ be an almost Hermitian Lie algebroid over a smooth manifold $M$. The second fundamental form section $B^{0,1}$ of $E^{0,1}$ is given by the formulas
\begin{displaymath}
{\rm Im}\,B^{0,1}(s_1,s_2)={\rm Re}\,B^{0,1}(s_1,J_E(s_2))\,\,{\rm and}\,\,g({\rm Re}\,B^{0,1}(s_1,s_2),s_3)=
\end{displaymath}
\begin{equation}
\label{m16}
=\frac{1}{16}\left[g(N_{J_E}(s_1,s_2),s_3)+g(N_{J_E}(s_2,J_E(s_3)),J_E(s_1))-g(N_{J_E}(J_E(s_3),s_1),J_E(s_2))\right]
\end{equation}
where $s_1,s_2,s_3\in\Gamma(E)$.
\end{proposition}
By direct calculations we have the following equivalent form of the second formula of \eqref{m16}:
\begin{equation}
\label{m17}
g({\rm Re}\,B^{0,1}(s_1,s_2),s_3)=-\frac{1}{8}\left[(D_{J_{E}(s_1)}\Phi)(s_2,s_3)+(D_{s_1}\Phi)(J_E(s_2),s_3)\right].
\end{equation}
Also, it is easy to see that
\begin{equation}
\label{m18}
B^{0,1}(J_E(s_1),J_E(s_2))=-B^{0,1}(s_1,s_2), 
\end{equation}
which yields the following interpretation of the Nijenhuis tensor, namely
\begin{equation}
\label{m19}
N_{J_E}=16{\rm Re}\,(alt\,B^{0,1}).
\end{equation}
\begin{corollary}
\label{curvsection}
Let $(E,\rho_E,[\cdot,\cdot]_E,g)$ be an almost Hermitian Lie algebroid over a smooth manifold $M$. The mean curvature section $H^{0,1}$ of $E^{0,1}$ is zero. The second fundamental form section $B^{0,1}$ of $E^{0,1}$ vanishes iff $(E,\rho_E,[\cdot,\cdot]_E,g)$ is Hermitian.
\end{corollary}
\begin{proof}
Taking into account the definition \eqref{m13} of $H^{0,1}$ with an orthonormal basis of the form $\{e_a,J_E(e_a)\}$, $a=1,\ldots,m$ and using \eqref{m18}, we obtain $H^{0,1}=0$. The formulas \eqref{m19} and \eqref{m16} implies that $B^{0,1}=0$ iff $N_{J_E}=0$.
\end{proof}

\noindent 
Cristian Ida\\
Department of Mathematics and Informatics, University Transilvania of Bra\c{s}ov\\
Address: Bra\c{s}ov 500091, Str. Iuliu Maniu 50,  Rom\^{a}nia\\
 email:\textit{cristian.ida@unitbv.ro}
\medskip

\noindent 
Paul Popescu\\
Department of Applied Mathematics, University of Craiova\\
Address: Craiova, 200585,  Str. Al. Cuza, No. 13,  Rom\^{a}nia\\
 email:\textit{paul$_{-}$p$_{-}$popescu@yahoo.com}

\end{document}